\documentclass[11pt]{article}
\usepackage[english]{babel}
\usepackage{amsmath,amsthm,amssymb}
\usepackage[all]{xy}
\usepackage{setspace}
\usepackage[top=1.3in, bottom=1.6in, left=1.3in, right=1.3in]{geometry}
\usepackage{color}
\definecolor{grau}{rgb}{0.5,0.5,0.5}
\usepackage[colorlinks, linkcolor=grau, citecolor=grau, urlcolor=grau]{hyperref}
\usepackage{breakurl}
\usepackage{bibspacing}
\usepackage{leftidx}
\newtheorem{theorem}{Theorem}[section]
\newtheorem{lemma}[theorem]{Lemma}
\newtheorem{proposition}[theorem]{Proposition}

\newtheorem{corollary}[theorem]{Corollary}

\numberwithin{equation}{section}

\newcommand {\spe} {{\textnormal{Spec}}}

\newcommand {\N} {{N_{K/\mathbb Q}}}

\newcommand {\im} {{\textnormal{im}}}
\newcommand {\OK}  {{\mathcal O_{K}}}

\newcommand {\OSt}  {{\mathcal O^{\times}_{S}}}

\newcommand {\QQ}  {{\mathbb Q}}
\newcommand {\ZZ}  {{\mathbb Z}}
\newcommand {\ord} {{\textnormal{ord}}}

\newcommand {\PP}  {{\mathbb P}}

\newcommand {\ts}  {t}

\newcommand {\ls}  {U}

\newcommand {\kk} {\kappa}
\newcommand {\CC} {\mathbb C}
\newcommand {\ccc} {c}
\newcommand {\nn} {\gamma}

\newcommand {\Sb} {T}
\newcommand {\Ov}{\mathcal O_v}

\date{}

\begin{document}
\author{Rafael von K\"anel}
\title{On Szpiro's Discriminant Conjecture}
\maketitle

\begin{abstract}
We consider generalizations of Szpiro's classical discriminant conjecture to hyperelliptic curves over a number field $K$, and to smooth, projective and geometrically connected curves $X$ over $K$ of genus at least one. The main results give effective exponential versions of the generalized conjectures for some curves, including all curves $X$ of genus one or two. We obtain in particular exponential versions of Szpiro's classical discriminant conjecture for elliptic curves over $K$. In course of our proofs we establish explicit results for certain Arakelov invariants  of hyperelliptic curves (e.g. Faltings' delta invariant)  which are of independent interest. The proofs use the theory of logarithmic forms and Arakelov theory for arithmetic surfaces. 
\end{abstract}

\section{Introduction}
\label{secintro}
Let $K$ be a number field and let $E$ be an elliptic curve over $K$. We denote by  $\Delta_E$ and by $N_E$ the norm from $K$ to $\QQ$  of the minimal discriminant ideal and of the conductor ideal of $E$ over $K$ respectively. In 1982 Szpiro conjectured that there are constants $c,\kappa$, depending only on $K$, such that
\begin{equation}\label{szpiroconj}
\Delta_E\leq c N_E^{\kappa}.
\end{equation}
This inequality is related to several conjectures in number theory. We mention the $abc$-conjecture $(abc)$ of Masser-Oesterl\'e, Vojta's conjectures in dimension one, ``good" effective versions of the Mordell conjecture and the arithmetic Bogomolov-Miyaoka-Yau inequality.

In this paper we consider two generalizations of (\ref{szpiroconj}). The first generalization, Conjecture $(\Delta)$, is to  hyperelliptic curves $C$ over $K$ of genus $\geq 1$, and the second generalization, Conjecture $(D)$, is to arbitrary smooth, projective and geometrically connected curves $X$ over $K$ of genus $\geq 1$. 
Our main results (see Theorem \ref{thm1}, \ref{thm2} and \ref{thm3}) give in particular effective exponential versions of the generalized conjectures for all curves $C$ and for all curves $X$ of genus one or two. For instance, Theorem \ref{thm3} provides completely explicit constants $c,\kappa$, depending only on $K$, such that
\begin{equation}\label{eq:expszpiro}
\log \Delta_E\leq cN_E^{\kappa}.
\end{equation}
As far as we know this is the first unconditional exponential version of (\ref{szpiroconj}), even for $K=\QQ$. This might be surprising in view of the exponential versions of the related $abc$-conjecture, established by Stewart-Tijdeman \cite{stti:abc} and Stewart-Yu \cite{styu:abc1,styu:abc2}. But in Section \ref{secstatement} we shall see that the classical links $(abc)\Rightarrow (\ref{szpiroconj})$ do not work any more with exponential versions.  Thus it seems not possible to deduce (\ref{eq:expszpiro}) for arbitrary $E$ over $\QQ$ from the exponential versions obtained in \cite{stti:abc,styu:abc1,styu:abc2}.

In course of our proofs we establish several new results for (Arakelov) invariants of hyperelliptic curves over $K$ which may be of independent interest. 
For instance,  if $X_\CC$ is a compact connected Riemann surface of genus two, then we show that Faltings' delta invariant $\delta(X_\CC)$ of $X_\CC$ satisfies $-186\leq\delta(X_\CC).$ This completely explicit lower bound is used in \cite{javk:smallpoints} to establish the genus two case of Szpiro's small points conjecture.
Furthermore, our proofs give in addition a new link between the $abc$-conjecture and (\ref{szpiroconj}). This link is more conceptual compared to the classical connection through Mordell equations, and it provides a new tool to apply the theory of logarithmic forms, or the $abc$-conjecture, to several problems in Diophantine geometry, see for example \cite{javk:smallpoints},  \cite{rvk:height} and \cite{rvk:ed}.

We also motivate Conjecture $(\Delta)$ and $(D)$. 
Theorem \ref{thm4} shows that $(abc)$ implies parts of $(\Delta)$ and $(D)$. In addition, we prove that the height conjecture of Frey \cite{frey:linksulm} implies $(D)$ for all $X$ over $K$ which have semi-stable reduction over the ring of integers of $K$. 

Next, we describe the main tools used in our proofs.  Theorem \ref{thm1} is based on the effective Shafarevich theorem for hyperelliptic curves over $K$ established in \cite{rvk:hyperelliptic}.
To prove Theorem \ref{thm2} we first slightly refine the method developed by Par{\v{s}}in \cite{parshin:shafarevich}, Oort \cite{oort:shafarevich} and de Jong-R\'emond \cite{jore:shafarevich}. 
This leads  to an exponential version of Frey's height conjecture for certain Jacobian varieties, and then we use results of Faltings \cite{faltings:arithmeticsurfaces} in Arakelov geometry. 
To get here effective estimates for hyperelliptic curves, we use formulas of Bost \cite{bost:genus2} and de Jong \cite{dejong:riemanninvariants,dejong:weierstrasspoints}. 
In the  important case of genus one curves we employ a different method to obtain sharper bounds, see Theorem \ref{thm3}. 
Here we use the theory of logarithmic forms, results of Faltings \cite{faltings:arithmeticsurfaces}, Liu-Lorenzini-Raynaud \cite{lilora:genusone} and Pesenti-Szpiro \cite{pesz:discriminant}, and the classification of Kodaira-N\'eron \cite{neron:models}. 
Theorem \ref{thm1}, \ref{thm2} and \ref{thm3} depend all ultimately on the theory of logarithmic forms, 
and we obtain Theorem \ref{thm4} on working out precisely this dependence. To deduce the corollaries we apply results for discriminants of Liu \cite{liu:conductor} and of Saito \cite{saito:genus2}. 

The plan of the article is as follows. In Section \ref{secconjectures} we define discriminants of curves $C$ and $X$, and we generalize Szpiro's conjecture. We also give some motivation for the generalizations. 
In Section \ref{secstatement} we state and discuss our results. 
Then in Section \ref{secheights} and Section \ref{secdiscriminants}  we use algebraic and Arakelov geometry, and we relate discriminants and heights of $C$ and $X$ to invariants which control the ramification of the curves over the projective line. In Section \ref{secunits} we apply the theory of logarithmic forms to estimate effectively the height of the solutions of unit equations in field extensions of $K$. Finally, in Section \ref{secproofs}, we give the proofs of the theorems and the corollaries.

Throughout this paper we shall use the following notations and conventions. By a curve $X$ over $K$ we mean a smooth, projective and geometrically connected curve $X\to \spe(K)$. We say that $X$ is semi-stable if it has semi-stable reduction over the ring of integers $\OK$ of $K$.  
We identify a prime ideal of $\OK$ with the corresponding finite place $v$ of $K$ and vice versa, we write $N_v$ for the number of elements in the residue field of $v$ and we denote by $v(\mathfrak a)$ the order of $v$ in a fractional ideal $\mathfrak a$ of $K$. By the Faltings height of an abelian variety  $A$ over $K$ we mean the absolute stable height of $A$, defined in \cite[p.354]{faltings:finiteness}.   
Finally, by $\log$ we mean the principal value of the natural logarithm and we define the maximum of the empty set and the product taken over the empty set as $1$.

\section{Discriminant Conjectures}\label{secconjectures}

Let $K$ be a number field and let $g\geq 1$ be an integer. In this section we give two generalizations of Szpiro's discriminant conjecture. The first is for hyperelliptic curves over $K$ of genus $g$, and the second is for arbitrary curves $X$ over $K$ of genus $g$. We also provide some motivation for the generalized discriminant conjectures.

We define a hyperelliptic curve over $K$ of genus $g$ as a curve $X$ over $K$ of genus $g$ that admits a finite morphism $X\to \PP^1_K$ of degree 2, where $\PP^1_K$ is the projective line over $K$. For example,  all elliptic curves over $K$ and all genus two curves $X$ over $K$ are hyperelliptic curves over $K$, see \cite[Proposition 7.4.9]{liu:ag}. Let $C$ be a hyperelliptic curve over $K$ of genus $g$ and  $\mathcal O\subseteq K$ a Dedekind domain with field of fractions $K$. We observe that the function field $K(C)$ of $C$ satisfies $K(C)=K(x)[y]$, where
\begin{equation*}
y^{2}+f_2(x)y=f(x),\ \ f, f_2\in \mathcal O[x]
\end{equation*}
and $2g+1\leq \max \bigl( 2 \deg f_2, \deg f\bigl)\leq 2g+2$. The normalization $\mathcal W(f,f_2)$  of the $\spe(\mathcal O)$-scheme $\spe(\mathcal O[x])\cup \spe(\mathcal O[1/x])$ in  $K(C)$ is called a Weierstrass model of $C$ over $\spe(\mathcal O)$. We remark that $\mathcal W(f,f_2)$ is a Weierstrass model in the sense of Liu \cite[Definition 6]{liu:models} and we note that \cite[Definition 6]{liu:models} generalizes the well-known notion of Weierstrass models over $\spe(\mathcal O)$ for elliptic curves. To simplify notation we write $\mathcal W(f)$ for $\mathcal W(f,f_2)$ if $f_2=0$. The discriminant $\Delta$ of $\mathcal W(f,f_2)$ is defined by
\begin{equation*}
\Delta=\begin{cases}
2^{4g}\Delta(f_0) & \textnormal{if } \deg f_0=2g+2\\
2^{4g}a_0^2\Delta(f_0) & \textnormal{otherwise},
\end{cases}
\end{equation*}
where $f_0=f+f_2^2/4$ has leading coefficient $a_0$ and discriminant $\Delta(f_0)$.

Let $v$ be a finite place of $K$ and let $\mathcal O_{K,v}\subset K$ be the local ring at $v$. We define $n_v=\min(v(\Delta))$
with the minimum taken over all discriminants $\Delta$ of Weierstrass models of $C$ over $\spe(\mathcal O_{K,v})$. The  integer $n_v$ is non-negative, and $n_v$ is zero if and only if $C$ has good reduction at $v$. We define the minimal discriminant $\Delta_C$ of $C$ by
$$
\Delta_C=\prod N_v^{n_v}
$$
with the product extended over all finite places $v$ of $K$. We note that this explicit definition of $\Delta_C$ is intrinsic. A result of Liu \cite[Proposition 8 (c)]{liu:models} implies that if $C$ is an elliptic curve $E$ over $K$, then $\Delta_C$ is the usual minimal discriminant $\Delta_E$ of $E$. 

We now define a discriminant $D_X$ for any curve $X$ over $K$ of genus $g$. Let $v$ be a finite place of $K$. To define a local discriminant exponent $\delta_v$ of $X$ at $v$ we follow Deligne \cite{deligne:quillen} and Saito \cite{saito:conductor}. The definition uses Mumford's isomorphism \cite[Theorem 5.10]{mumford:projectivevarieties} which is based on Grothendieck's relative Riemann-Roch theorem. We denote by $\mathcal O_v$ the valuation ring of the completion $K_v$ of $K$ at $v$ and we write $B_v=\spe(\Ov)$. Let $X_v$ be the base change of $X$ to $K_v$ and let $\rho:\mathcal X_v\to B_v$ be the minimal regular model of $X_v$ over $B_v$. The generic fiber $\rho_\eta:X_v\to \spe(K_v)$ of $\rho:\mathcal X_v\to B_v$ is smooth and thus Deligne \cite[p.C1]{deligne:quillen}, or Saito \cite[p.154]{saito:conductor}, gives an isomorphism of invertible sheaves on $\spe(K_v)$
\begin{equation*}
D:\det R{\rho_\eta}_*(\omega_{X_v/K_v}^{\otimes 2})\rightarrow (\det R{\rho_\eta}_*(\omega_{X_v/K_v}))^{\otimes 13}
\end{equation*}
which is unique up to a sign and which commutes with arbitrary base change. Here  $\omega_{Y/Z}$ denotes the relative canonical sheaf of a projective flat morphism $\pi:Y\to Z$ of regular schemes with $Z$ Noetherian. 
A definition and basic properties of $\det R\pi_*$ are given for example in \cite[Section 5]{lilora:genusone}.
We identify the invertible sheaves
$$\det R\rho_*(\omega_{\mathcal X_v/B_v}^{\otimes 2}) \textnormal{ and } (\det R\rho_*(\omega_{\mathcal X_v/B_v}))^{\otimes 13}$$
on $B_v$ with their canonical images, see \cite[Lemma 5.3]{lilora:genusone}, in
$$\det R{\rho_\eta}_*(\omega_{X_v/K_v}^{\otimes 2}) \textnormal{ and } (\det R{\rho_\eta}_*(\omega_{X_v/K_v}))^{\otimes 13}$$ respectively. Then \cite[Lemma 5.3]{lilora:genusone} gives a non-zero $\delta\in K_v$ such that
\begin{equation*}
D\bigl(\det R\rho_*(\omega_{\mathcal X_v/B_v}^{\otimes 2})\bigl)=\delta(\det R\rho_*(\omega_{\mathcal X_v/B_v}))^{\otimes 13}.
\end{equation*}
We take $\delta_v=v(\delta)$. It turns out that $\delta_v$ is a non-negative integer and if $\mathcal X_v$ is semi-stable, then $\delta_v$ coincides with the number of singular points of the geometric special fiber of $\rho:\mathcal X_v\to B_v$. In particular, if $\rho:\mathcal X_v\to B_v$ is smooth, then $\delta_v=0$. Therefore we can define the discriminant $D_X$ of $X$ by
\begin{equation*}
D_X=\prod N_v^{\delta_v} 
\end{equation*}
with the product taken over all finite places $v$ of $K$. We mention that if $X$ is an elliptic curve $E$ over $K$, then Proposition \ref{propde} (ii) gives $D_X=\Delta_E$.

To define a conductor $N_X$ for any curve $X$ over $K$ of genus $g$ we take again a finite place $v$ of $K$. Let $f_v$ be the usual conductor exponent at $v$ of the Jacobian variety $J=\textnormal{Pic}^{0}(X)$ of $X$, see \cite[p.575]{lorosi:conductor}. 
We set $a_v=1$ if $X$ has bad reduction at $v$ and $J$ has good reduction at $v$, and $a_v=0$ otherwise. Then we define the conductor $N_X$ of $X$ by 
\begin{equation*}
N_X=\prod N_v^{f_v+a_v}
\end{equation*}
with the product taken over the finite places $v$ of $K$. If $X$ is an elliptic curve $E$ over $K$, then $a_v=0$ for all finite places $v$ of $K$ and we get that $N_X$ is the usual conductor $N_E$ of the abelian variety $E$ over $K$.

We now generalize  Szpiro's classical discriminant conjecture  for elliptic curves over $K$, see \cite[Conjecture 1]{szpiro:seminaire88} or (\ref{szpiroconj}), to more general curves.

\vspace{0.2cm}
\noindent{\bf Conjecture ($\Delta$)}
\emph{There are constants $\ccc,\kk$, depending only on $K$ and $g$, such that if $C$ is a hyperelliptic curve over $K$ of genus $g$, then $\Delta_C\leq \ccc N_C^{\kk}.$\medskip}

We mention that  Lockhart \cite[Conjecture 4.1]{lockhart:discriminant} proposes a slightly stronger conjecture for hyperelliptic curves over $K$ of genus $g$ with a $K$-rational Weierstrass point. Therein the exponent $\kk$ depends only on $g$.

\vspace{0.2cm}
\noindent{\bf Conjecture ($D$)}
\emph{There exist constants $\ccc,\kk$, depending only on $K$ and $g$, such that if $X$ is a curve over $K$ of genus $g$, then $D_X\leq \ccc N_X^{\kk}.$\medskip}

To keep the analogy with the elliptic case we stated the above conjectures in terms of the somewhat crudely defined conductor $N_X$. It would be more natural to use
$M_X=\prod  N_v$ with the product taken over all finite places $v$ of $K$ where $X$ has bad reduction. However, Lockhart-Rosen-Silverman \cite[Theorem 0.1]{lorosi:conductor} give an effective constant $\kappa'$, depending only on $g$ and on the degree of $K$ over $\QQ$, such that $M_X \leq N_X\leq M_X^{\kappa'}.$ Hence one can replace $N_X$ by $M_X$ in $(\Delta)$ and $(D)$ without changing the statements. 

Conjecture $(D)$ for hyperelliptic curves and Conjecture ($\Delta$) are not unrelated. For example, if $X=E=C$ is an elliptic curve over $K$, then the identities $$D_X=\Delta_E=\Delta_C, \ \ \ N_X=N_E=N_C,$$  show that $(D)$, Szpiro's classical conjecture (\ref{szpiroconj}) and ($\Delta$) are all equivalent for elliptic curves over $K$\footnote[1]{In August 2012, Shinichi Mochizuki posted a preprint on his homepage in which he claims to prove these conjectures for elliptic curves over $K$.}. But for hyperelliptic curves $C$ over $K$ of genus $g\geq 2$ the situation is more complicated. Kausz \cite[p.44]{kausz:discriminant} introduced a further discriminant exponent $\Lambda_v$ of $C$ at a finite place $v$ of $K$ with $2\nmid N_v$. To define $\Lambda_v$ for $v$ with $2\mid N_v$ one can use Maugeais \cite[Proposition 2.7]{maugeais:discriminant}. The exponent $\Lambda_v$  is closely related to $n_v$ and  Cornalba-Harris, de Jong, Kausz, Liu, Maugeais, Saito, Ueno and Yamaki compared $\Lambda_v$ to $\delta_v$; for an overview on these comparison results and precise references we refer to \cite[p.274]{dejong:weierstrasspoints}. In the semi-stable case there are explicit inequalities relating $g$, $\delta_v$ and $\Lambda_v$ and it should be possible to show that ($D$) and ($\Delta$) are equivalent for semi-stable $C$. However, it remains an interesting project to find relations between the discriminant conjectures for non-semi-stable $C$.

In the remaining of this section we provide some motivation for our conjectures. We note that geometric analogues of $(\Delta)$ and $(D)$ are already established in many cases. For example, geometric versions of $(D)$ are given by Arakelov \cite{arakelov:shafarevich} and Szpiro \cite{szpiro:parsin,szpiro:deux}, see also the more recent article of Beauville \cite{beauville:szpiro} and the preprint of Kim \cite{kim:szpiroconj}. Furthermore, there exists a geometric analogue of $(\Delta)$ due to Bogomolov-Katzarkov-Pantev \cite{bokapa:szpiro} and Pesenti-Szpiro \cite{pesz:discriminant} obtained a fairly general elliptic version of $(\Delta)$.

We now give propositions which provide further motivation for ($\Delta$) and $(D)$.  These propositions follow more or less directly from suitable interpretations of known results and thus we include the references in the statements. Let $h_F(A)$ be the Faltings height of an abelian variety $A$ over $K$. We define the conductor $N_A$ of $A$ by $N_A=\prod N_v^{f_v}$ with the product taken over all finite places $v$ of $K$, where $f_v$ is the usual  conductor exponent at $v$ of $A$.  The following conjecture  is a weaker version of Frey's conjecture \cite[p.39]{frey:linksulm}. 

\vspace{0.2cm}
\noindent{\bf Conjecture ($H$)} (\cite{frey:linksulm}).
\emph{There are constants $c,\kappa$, depending only on $K,g$, such that if $A$ is an abelian variety over $K$ of dimension $g$, then $h_F(A)\leq \kappa\log N_A+c.$\medskip}

Frey \cite{frey:linksulm} motivates this conjecture in the case $g=1$. In particular, he proves a function field analogue of $(H)$ for elliptic curves. We mention in addition that if $g\geq 2$, then geometric analogues of $(H)$ were established by Faltings \cite{faltings:arakelovtheorem}, by Deligne \cite{deligne:monodromie} and by Kim \cite{kim:abc}. 
To relate $(H)$ to $(D)$ we observe that the conductor of the Jacobian of $X$ divides $N_X$. Hence Proposition \ref{propd} (iv), which is based on fundamental results of Faltings \cite{faltings:arithmeticsurfaces}, implies the following proposition.
\begin{proposition}[\cite{faltings:arithmeticsurfaces}]
Conjecture $(H)$ for semi-stable Jacobian varieties over $K$ of dimension $g$ implies $(D)$ for semi-stable curves over $K$ of genus $g$.
\end{proposition}

We next consider the case of potential good reduction. Let $C$ be a hyperelliptic curve over $K$ of genus $g$, with potential good reduction at a finite place $v$ of $K$. Liu \cite[Remarque 13]{liu:models} showed that if the residue characteristic of $v$ is at least $2g+2$, then the discriminant exponent $n_v$ of $C$ satisfies $n_v\leq\kappa(g)=2(2g+1)([(g+1)/2]+1).$ Here $[(g+1)/2]$ denotes the smallest integer which is at least $(g+1)/2$. We deduce the following result. 

\begin{proposition}[\cite{liu:models}]\label{propliu}
Suppose $C$ is a hyperelliptic curve over $K$ of genus $g$. If $C$ has everywhere potential good reduction and any rational prime divisor of $N_C$ is at least $2g+2$, then $\Delta_C\vert M_C^{\kappa(g)}.$
\end{proposition}

The exponent $\kappa(g)$ is optimal for any $g\geq 2$, since $\Delta_C=M_C^{\kappa(g)}$ holds for infinitely many curves $C$, see Liu \cite[p.4596]{liu:models}. Proposition \ref{propliu} and Proposition \ref{propd} (iii) show that if a genus two curve $X$ over $K$ has everywhere potential good reduction and any rational prime divisor of $N_X$ is at least seven, then  $D_X\vert M_X^{30}.$ In the genus one case one can remove the technical assumption on the residue characteristics. Pesenti-Szpiro \cite{pesz:discriminant} derived from the classification of Kodaira-N\'eron \cite{neron:models}  that any elliptic curve $E$ over $K$ with integral $j$-invariant satisfies $\Delta_E \vert N_E^5$. This together with Proposition \ref{propde}, which is based on a theorem of Liu-Lorenzini-Raynaud \cite{lilora:genusone}, gives the following proposition.
\begin{proposition}[\cite{neron:models}, \cite{pesz:discriminant}, \cite{lilora:genusone}]
If $X$ is a genus one curve over $K$ with everywhere potential good reduction, then $D_X\vert N_X^5$.
\end{proposition}

On considering, for example, elliptic curves with $j$-invariant zero we see that the exponent 5 is optimal. To conclude our discussion we remark that results of Serre \cite{serre:representations} and Mestre-Oesterl\'e \cite{meoe:weilcurvediscriminants} imply that if $X$ is a genus one curve over $\QQ$ with  conductor a rational prime, then $D_X\vert N_X^{5}$.  

\section{Statement of the results}\label{secstatement}

In this section we state and discuss our main results. Let $K$ be a number field and let $g\geq 1$ be an integer. In the sequel $c_i,\kappa_i,$ $i\geq 1$, denote effective constants depending only on $K$ and $g$. They shall be given in a more explicit form at the end of this section.

The following theorem  establishes an effective exponential version of Conjecture ($\Delta$).

\begin{theorem}\label{thm1}
Suppose $C$ is a hyperelliptic curve over $K$ of genus $g$, with minimal discriminant $\Delta_C$ and conductor $N_C$. Then the following statements hold.
\begin{itemize}
\item[(i)] If  $C$ has a $K$-rational Weierstrass point, then $\log \Delta_C\leq \ccc_1N_C^{\kk_1}.$
\item[(ii)] In general $C$ satisfies $\log \Delta_C\leq \exp(\ccc_2(\log N_C)^{\kk_2}).$\end{itemize}\end{theorem}

Let $X$ be a curve over $K$ of genus $g$, with discriminant $D_X$ and conductor $N_X$. We say that $X$ is a cyclic cover of prime degree if there exists a finite morphism $X\to \mathbb P^1_K$ of prime degree which is geometrically a cyclic cover.  
In the sequel, $c'_3$ denotes a constant depending only on $K$ and $g$. 

\begin{theorem}\label{thm2}
If $X$ is a semi-stable cyclic cover of prime degree with genus $g\geq 2$, then $\log D_X\leq c'_3 N_X^{\kk_3}.$
\end{theorem}

So far the constant $c_3'$ is only effective for those $X$ in Theorem \ref{thm2} which are in addition hyperelliptic, see  (\ref{eq:thm2b}) for a discussion of the effectivity of $c_3'$.

Next, we  consider curves of genus one or two. We recall that if $X=E$ is an elliptic curve over $K$, then $D_X$ coincides with the usual minimal discriminant $\Delta_E$ of $E$. Hence we see that our next theorem establishes in particular an effective exponential version of Szpiro's classical discriminant conjecture which we stated in (\ref{szpiroconj}).

\begin{theorem}\label{thm3}
If $X$ has genus one, then $\log D_X\leq c_4N_X^{\kappa_4}.$
\end{theorem}

We discuss the relation of this result to exponential versions of the $abc$-conjecture over $\QQ$ of Stewart-Tijdeman \cite{stti:abc} and Stewart-Yu \cite{styu:abc1,styu:abc2}. In this discussion we denote by $a,b,c$  non-zero coprime integers. We suppose $a+b=c$ and we let $S_\QQ(a,b,c)$ be the product of all rational prime divisors of $abc$. The elliptic curve $E$ over $\QQ$ with affine plane model $y^2=x(x+a)(x-b)$ satisfies $N_{E}\leq 2^8S_\QQ(a,b,c)$ and  $\left|abc\right|^2\leq 2^8\Delta_{E}$, see \cite[p.257]{silverman:aoes}. 
Thus Theorem \ref{thm3} gives an effective absolute  constant $c_4'$ such that $$ \max(\left|a\right|,\left|b\right|,\left|c\right|)\leq \exp(c_4'S_\QQ(a,b,c)^{22}).$$ The exponent 22 in this exponential version of $(abc)$ is weaker compared to the exponents in \cite{stti:abc,styu:abc1,styu:abc2}. On the other hand, \cite{stti:abc,styu:abc1,styu:abc2} combined with the known links  $(abc)\Rightarrow (1)$ do not imply an upper bound for $\Delta_E$ in terms of $N_E$ for arbitrary elliptic curves $E$ over $\QQ$. The reason is that these classical links, see for example \cite[p.260]{silverman:aoes}, \cite[p.429]{bogu:diophantinegeometry}, are not compatible with exponential estimates. The problem is always of the following elementary kind: If $m,n$ are positive integers, then $m^2\leq mn$ implies $m\leq n$, but $m^2\leq \exp(mn)$ does not imply in general $m\leq \exp(n)$. This concludes our discussion.  

The following corollary is based on Theorem \ref{thm1} and Theorem \ref{thm2}, and on results of Liu and Saito for discriminants of genus two curves.

\begin{corollary}\label{corgenustwo}
Suppose $X$ has genus two. Then the following statements hold.
\begin{itemize}
\item[(i)] If $X$ is semi-stable or has a $K$-rational Weierstrass point, then $\log D_X\leq c_5N_X^{\kappa_5}.$
\item[(ii)] In general $X$ satisfies
$\log D_X\leq \exp(c_6(\log N_X)^{\kappa_6}).$
\end{itemize}
\end{corollary}

Before we state Theorem \ref{thm4}, we recall the $abc$-conjecture $(abc)$ of Masser-Oesterl\'e \cite{masser:abc} over number fields.  Let $d$ be the degree of $K$ over $\QQ$. For any non-zero triple $\alpha,\beta,\gamma\in K$ we denote by $h(\alpha,\beta,\gamma)$ the usual absolute logarithmic Weil height (see \cite[1.5.4]{bogu:diophantinegeometry}) of $\alpha,\beta,\gamma$ viewed as a point in $\mathbb P^2(K)$. We define $$H_K=\exp(dh), \ \ S_K(\alpha,\beta,\gamma)=\prod N_v^{e_v}$$ with the product extended over all finite places $v$ of $K$ such that $v(\alpha)$, $v(\beta)$ and $v(\gamma)$ are not all equal, where $e_v=v(p)$ for $p$ the residue characteristic of $v$. We mention that Masser \cite{masser:abc} added the ramification index $e_v$ in the definition of the support $S_K$ to obtain a natural behavior of $S_K$ under field extensions. Let $D_K$ be the absolute value of the discriminant of $K$ over $\QQ$.

\vspace{0.3cm}
\noindent{\bf Conjecture ($abc$).}
\emph{For any integer $n\geq 1$, and any real $r,\epsilon>1$, there exists a constant $c$, depending only on $n,r,\epsilon$, such that if $K$ is a number field of degree $d\leq n$, and $\alpha,\beta,\gamma\in K$ are non-zero with $\alpha+\beta=\gamma$, then $H_K(\alpha,\beta,\gamma)\leq cS_K(\alpha,\beta,\gamma)^rD_K^\epsilon.$\bigskip}

In the sequel $(abc)^*$ refers to the following weaker conjecture. For any integer $n\geq 1$ there are constants $c,\kappa$, depending only on $n$, such that if $K$ is a number field with degree $d\leq n$ and $\alpha,\beta,\gamma\in K$ are non-zero with $\alpha+\beta=\gamma$, then $H_K(\alpha,\beta,\gamma)\leq c(S_K(\alpha,\beta,\gamma)D_K)^{\kappa}.$ 

\begin{theorem}\label{thm4}
Suppose $g\geq 2$. Then the following statements hold.
\begin{itemize}
\item[(i)]  $(abc)^*$ implies \textnormal{(}$\Delta$\textnormal{)} for hyperelliptic curves over $\QQ$ of genus $g$ with a $\QQ$-rational Weierstrass point.
\item[(ii)]  $(abc)^*$ implies $(D)$ for semi-stable curves over $K$ of genus $g$ which are cyclic covers of prime degree.
\item[(iii)] Suppose there are $r,\epsilon>1$ such that $(abc)$ holds for $n=6d,r,\epsilon$ with the constant $c$. Then there exists an effective constant $c'$, depending only on $d,r,\epsilon,c$, such that any genus one curve $X$ over $K$ satisfies $\log D_X\leq 6(dr+\epsilon)\log N_X+6\epsilon\log D_K+c'.$
\end{itemize}
\end{theorem}

In particular, (iii) gives that $(abc)^*$ implies ($D$) for genus one curves over $K$. One can derive from Lockhart's \cite[Proposition 4.3]{lockhart:discriminant} that $(abc)$ implies ($\Delta$) for hyperelliptic curves over $K$ with all Weierstrass points $K$-rational. Furthermore, the classical link  through Mordell equations, applied in \cite{lockhart:discriminant}, leads also to a version of Theorem \ref{thm4} (iii). More precisely,  the arguments in \cite[Section 4]{lockhart:discriminant} can be used to deduce the following. If $(abc)$ holds for $n=d,r,\epsilon$, then there is a constant $c''$, depending only on $K$, $r$, $\epsilon$, such that any genus one curve $X$ over $K$ satisfies
\begin{equation}\label{eq:lock}
D_X\leq c''N_X^{6dr}.
\end{equation}
This gives the better exponent $6dr$ of $N_X$, but also a constant $c''$ that depends in an unspecified 
way on $D_K$.  Masser \cite{masser:szpiro} showed that the exponent $6r>6$ is best possible for Szpiro's discriminant conjecture over $\QQ$. We remark that the exponents of $N_X$ in (iii) and (\ref{eq:lock}) depend on $d$, since $S_K$ is defined with exponents $e_v\leq d$ which do not appear in the definition of $N_X$.

Finally, we shall deduce from Theorem \ref{thm4} the following corollary.

\begin{corollary}\label{corlowgenus}
$(abc)^*$ implies $(D)$ for genus two curves over $K$ which are semi-stable, or which are defined over $\QQ$ and have a $\QQ$-rational Weierstrass point.
\end{corollary}

To discuss an application of our discriminant estimates we take again  a curve $X$ over $K$ of genus $\geq 1$. Let $B$ be the spectrum of the ring of integers of $K$. The discriminant $D_X$ of $X$ is related to other global invariants of the minimal regular model $\mathcal X\to B$ of $X$ over $B$. To define these invariants we take a closed point $v$ of $B$, and we let $B_{\bar{v}}$ be the strict localization of $B$ at a geometric point $\bar{v}$ of $B$ over $v$.  We denote by $X_{\bar{\eta}}$ and by $ X_{\bar{v}}$ the geometric generic and the special fiber of the base change $\mathcal X\times_B B_{\bar{v}}\to B_{\bar{v}}$ respectively. Let $\ell$ be a rational prime with $\ell\nmid N_v$, let $\chi$ be the $\ell$-adic Euler-Poincar\'e characteristic  and  let $\textnormal{Sw}_{B_{\bar{v}}}H^1(X_{\bar{\eta}},\mathbb Q_\ell)$ be the Swan conductor, see \cite[p.422]{bloch:eulerchar}. It turns out that the Artin conductor  
$
\textnormal{Art}_v=\chi(X_{\bar{\eta}})-\chi(X_{\bar{v}})-\textnormal{Sw}_{B_{\bar{v}}}H^1(X_{\bar{\eta}},\mathbb Q_\ell)
$
does not depend on the choice of $\ell$.  Bloch \cite[Lemma 1.2(i)]{bloch:conductor} leads to
$
-\textnormal{Art}_v=m_v-1+f_v,
$
where $m_v$ denotes the number of irreducible components of $X_{\bar{v}}$, counted up to multiplicities, and $f_v$ denotes the usual conductor exponent at $v$ of the Jacobian of $X$. Let $\delta_v$ be the discriminant exponent of $X$ at $v$ defined in Section \ref{secconjectures}. Saito \cite{saito:conductor} established
$
-\textnormal{Art}_v=\delta_v;
$
he reduced the problem to the semi-stable case proven by Deligne in \cite{deligne:quillen}.
Hence, we see that our results for $D_X$ give in addition estimates  for two other global invariants of $\mathcal X\to B$ in terms of $N_X$. Namely, a lower bound for the global Artin conductor $\sum \textnormal{Art}_v\log N_v$ and an upper bound for the global number of ``additional'' geometrically irreducible components $\sum (m_v-1)$. Here both sums are taken over all closed points $v$ of $B$. This concludes our discussion.

It remains to give the constants.  We conducted some effort to obtain constants reasonably close to the best that can be acquired with the present method of proof. However, to simplify their final form we freely rounded off several of the numbers occurring in our estimates. Let $h_K$ be the class number of the ring of integers of $K$. In Theorem \ref{thm1} we can take for example
$$\kappa_1=(8gd)^4\log (3h_K), \ \ c_1=(3D_K^{h_K})^{\kappa_1}, \ \ \kappa_2=6, \ \ c_2=k_0(4gd)^{15d}D_K^{10},$$
where $k_0$ is an effective absolute constant.
We note that $k_0$ comes from the effective upper bound in \cite[Theorem (ii)]{rvk:hyperelliptic} which is not yet completely explicit. In Theorem \ref{thm2} we can take for instance
\begin{equation*}
\kappa_3=96dg^4, \ \ c'_3=k_3'D_K^{24g^4},
\end{equation*}
where $k_3'$ is a constant depending only on $g$ and $d$.  In general $k_3'$ is not necessarily effective in terms of $g$, since it depends on lower bounds for Faltings' delta invariant which are not known to be effective. However, if $X$ is a semi-stable hyperelliptic curve over $K$ of genus $g\geq 2$, then
\begin{equation}\label{eq:thm2b}
\log D_X\leq c_3N_X^{d(4g)^8}, \ \ c_3=2^{2^{50}9^gd^2}D_K^{24g^4}.
\end{equation} 
In view of the applications in \cite{rvk:height} and \cite{rvk:ed} we conducted some effort to improve the constants for genus one curves. Theorem \ref{thm3} holds with
$$\kappa_4=22d, \ \ c_4=c(d)D_K^{4} \ \ \  \textnormal{or} \ \ \ \kappa_4=162d, \ \ c_4=5^{(18d)^2}D_K^5,$$
where $c(d)$ is an effective constant depending only on $d$. Finally, we mention that in Corollary \ref{corgenustwo} we can take
$$\kappa_5=8^8d^4\log(3h_K), \ \ c_5=\max(c_1(2),c_3(2)), \ \ \kappa_6=6, \ \ c_6=k_0(8d)^{15d}D_K^{10},$$
where $c_1(2)$ and $c_3(2)$ are the explicit constants $c_1$ and $c_3$ with $g=2$ respectively. This concludes our discussion of the constants.

\section{Heights}\label{secheights}

In this section we prove Proposition \ref{proph} and \ref{prophe}. They bound heights of certain curves in terms of an invariant which controls the ramification of the curve over the projective line. Proposition \ref{proph} deals with cyclic covers of prime degree of genus at least two, and Proposition \ref{prophe} deals with elliptic curves.

To state Proposition \ref{proph} we have to introduce some notation. Let $K$ be a number field, and let $X$ be a curve over $K$ of genus $g\geq 2$. We assume there exists a finite morphism $\varphi:X\to \mathbb P^1_K$  of prime degree which is geometrically a cyclic cover. Let $q$ be the degree of $\varphi$. For any  finite extension $L$ of $K$ we let $S(L,q)$ be the set of places of $L$ which divide $q$ or which divide a place of $K$ where $X$ has bad reduction. To ease notation we write $S=S(L,q)$ and we denote by $\mathcal O_S^{\times}$ the $S$-units in $L$. Let $h(f)$ be the usual absolute logarithmic Weil height of any polynomial $f$ with coefficients in $L$, see \cite[1.6.1]{bogu:diophantinegeometry}.  We write $\mu_{S}=\sup(h(\lambda), \lambda \in \mathcal O_S^\times \textnormal{ and } 1-\lambda \in \mathcal O_S^{\times})$. Let $\mathcal R$ be the set of field extensions $L$ of $K$ such that $L$ is the compositum of the fields of definition of four distinct (geometric) ramification points of $\varphi$. We define $$\mu_X=\max(1,\mu_{S(L,q)})$$ 
with the maximum taken over all fields $L\in \mathcal R$.
If $q=2$, then $X$ is a hyperelliptic curve over $K$, and in this case we define the height of a Weierstrass model $\mathcal W=\mathcal W(f,f_2)$  of $X$ over $\spe(\OK)$ by  $h(\mathcal W)=h(f+f_2^2/4)$ for $\OK$ the ring of integers of $K$. Let $h_F$ be the Faltings height, let $\nu=2g+1$ and let $N_X$ be  the conductor of $X$. We now can state the following proposition.
\begin{proposition}\label{proph}
Let $X$ be a curve over $K$ of genus $g\geq 2$. Suppose there is a finite morphism $\varphi:X\to \mathbb P^1_K$ of prime degree $q$ which is geometrically a cyclic cover. Then the following statements hold.
\begin{itemize}
\item[(i)] The Jacobian $J$ of $X$ satisfies $h_F(J)\leq 2^{2^{22}9^g}\mu_X.$
\item[(ii)] If $K=\QQ$, $q=2$ and $X$ has a $\QQ$-rational Weierstrass point, then there is a Weierstrass model $\mathcal W$ of $X$ over $\spe(\ZZ)$ that satisfies the following inequality  $h(\mathcal W)\leq 6\nu(\nu-1)\mu_X+12g\nu\log N_X+6\nu^2.$
\end{itemize}
\end{proposition}

We remark that the exponent $2^{2^{22} 9^g}$ in statement (i) can be replaced with $2^{3362·g^38^g}$. This exponent comes from height comparisons of R\'emond \cite{remond:rational} which hold for arbitrary curves over $K$. It seems possible to improve these height comparisons in our special case of cyclic covers.
Further, we mention that the assumption  $K=\QQ$ in statement (ii) is used to get an estimate which is linear in terms of $\log N_X$. For this our method requires  a  fundamental unit system with sufficiently small height, and in arbitrary number fields such a system does not always exist.
On the other hand, it seems possible (but technically extensive) to remove the Weierstrass point assumption in (ii) on going into the proof of \cite[Theorem (ii)]{rvk:hyperelliptic}.

Before we prove Proposition \ref{proph} we consider the genus one case. Let $E$ be an elliptic curve over a number field $K$ and let $K(E[2])$ be the field of definition  of the $2$-torsion points of $E$. We denote by $T$  the places of $K(E[2])$ which divide 2 or a finite place of $K$ where $E$ has not potential good reduction. Let $\mathcal O_T^{\times}$ be the $T$-units in $K(E[2])$. In the proof of Proposition \ref{prophe} below, we shall see that the quantity $$\mu_E=\max(h(\lambda), \lambda \in \mathcal O_T^\times \textnormal{ and } 1-\lambda \in \mathcal O_T^{\times})$$
controls ``the ramification of $E$ over the projective line''.

\begin{proposition}\label{prophe}
Any elliptic curve $E$ over $K$ satisfies $h_F(E)\leq \frac{1}{2}\mu_{E}+4$.
\end{proposition}

On considering the infinite family of curves $E$ with $K(E[2])=K$ we see that the coefficient $1/2$ in Proposition \ref{prophe} is best possible. On the other hand, the constant 4 therein is not optimal and it can be improved up to a certain extent.

\begin{proof}[Proof of Proposition \ref{prophe}]
Let $E$ be an elliptic curve over $K$ and write $L=K(E(2))$. We denote by $O$ the Weil divisor on $E$ corresponding to the zero section of $E$. Let $P_1,P_2,P_3$ be the three non-zero rational $2$-torsion points of the base change $E_L$ of $E$ to $L$.  We denote by $x$ a global section  of the line bundle $\mathcal O_E(2O)$ on $E_L$, where $\mathcal O_E(2O)$ is associated to the divisor $2O$ on $E_L$. The classical Riemann-Roch theorem for curves gives $x$ such that $1,x$ generate $\mathcal O_E(2O)$ and such that the $j$-invariant $j_E$ of $E$ satisfies
\begin{equation}\label{eq:j}
j_E=2^8\frac{((1-\lambda)^2+\lambda)^3}{\lambda^2(1-\lambda)^2}, \ \ \lambda=\frac{x(P_1)-x(P_2)}{x(P_1)-x(P_3)}\in L.
\end{equation}

To bound $h_F(E)$ in terms of $h(\lambda)$ we let $v$ be a place of $L$ and we take an exponential valuation $\left| \ \right|$ in $v$. If $v$ is infinite, then we deduce
$$\frac{\left|(1-\lambda)^2+\lambda\right|^3}{\left|\lambda^2(1-\lambda)^2\right|} 
\leq 2^3\frac{\max(\left|1-\lambda\right|^6,\left|\lambda\right|^3)}{\left|\lambda^2(1-\lambda)^2\right|}\leq 2^7 \max(\left|\lambda\right|^2,\left|\lambda\right|^{-2},\frac{\left|\lambda\right|}{\left|1-\lambda\right|^2})$$
and if $v$ is finite, then we get the same estimates but without the factors $2^3$ and $2^7$.
Let $H=\exp(h)$ be the  absolute exponential Weil height. Then the inequality $H(\lambda(1-\lambda)^{-2})\leq 4H(\lambda)^2$ and (\ref{eq:j}) lead to $H(j_E)\leq 2^{17}H(\lambda)^6$, and therefore $h_F(E)\leq 1/12h(j_E)+2.37$, obtained in \cite[Lemma 7.9]{gare:periods}, gives 
\begin{equation}\label{eq:hs}
h_F(E)\leq 1/2h(\lambda)+3.36. 
\end{equation}

To estimate $h(\lambda)$ in terms of $\mu_E$ we take a finite place $v$ of $L$ with $2\nmid N_v$. If $v(\lambda)\geq 1$, then $v(1-\lambda)=0$ 
and therefore (\ref{eq:j}) 
shows that $v(j_E)$ is negative which together with \cite[p.197]{silverman:aoes} implies that $v\in T$. If $v(\lambda)\leq -1$, then $v(\lambda^{-1})\geq 1$ and thus on replacing $P_2$ by $P_3$ in (\ref{eq:j}) we get again that $v\in T$. 
Hence we see that $\lambda\in L$ is a $T$-unit. The equality $$(1-\lambda)^{2}+\lambda=\lambda^2+(1-\lambda)$$
shows that $j_E$ is symmetric in $\lambda$ and $1-\lambda$ and thus we deduce that $1-\lambda\in L$ is a $T$-unit as well. We conclude that
$h(\lambda)\leq \mu_E$ and then the statement follows from (\ref{eq:hs}). This completes the proof of Proposition \ref{prophe}.\end{proof}

The rational function $x$ on $E$, which appears in the proof, defines a finite morphism $\varphi:E\to \mathbb P^1_K$ of degree two which ramifies exactly in the $2$-torsion points of $E$. This shows that if $X=E$ and $\mu_X$ is defined as above with respect to the cyclic cover $\varphi:X\to \mathbb P^1_K$ of degree two, then $\mu_E\leq \mu_X$.

In the remaining of this section we prove Proposition \ref{proph}. To obtain (i) we apply the method of Par{\v{s}}in \cite{parshin:shafarevich}, Oort \cite{oort:shafarevich} and de Jong-R\'emond \cite{jore:shafarevich}, 
and to show (ii) we use the strategy of \cite[Proposition 5.2 (i)]{rvk:hyperelliptic}.

\begin{proof}[Proof of Proposition \ref{proph}] By assumption, the curve $X$ over $K$ has genus $g\geq 2$, and there exists a finite morphism $\varphi:X\to \mathbb P^1_K$ of prime degree $q$ which is geometrically a cyclic cover.

To show (i) we follow the arguments of de Jong-R\'emond in \cite[Section 4]{jore:shafarevich}. Let $\bar{K}$ be an algebraic closure of $K$ and let $X_{\bar{K}}$ be the base change of $X$ to $\bar{K}$. Let $\Lambda$ be  the set of cross ratios of any four distinct (geometric) branch points of $\varphi$, see \cite[Section 2]{jore:shafarevich} in which $\Lambda$ is denoted by $\mathcal B$. We define $h_\Lambda=\max(1,h(\lambda))$ with the maximum taken over all $\lambda\in \Lambda$. As in \cite[Section 4]{jore:shafarevich} we see that the function field of $X_{\bar{K}}$ takes the form $$\bar{K}(x)[y]/\bigl(y^q-f(x)\bigl),\ \ f\in \bar{K}[x], \ \ h(f)\leq 7gh_\Lambda.$$
We choose $\lambda\in \Lambda$ and we observe that $1-\lambda$ and $\lambda^{-1}$ are both in $\Lambda$. Therefore  \cite[Proposition 2.1]{jore:shafarevich}  shows $\lambda\in \OSt \textnormal{ and } 1-\lambda\in \OSt$. This gives that $h(\lambda)\leq \mu_S$. Hence we see that $h_\Lambda\leq \mu_X$, since $K(\lambda)\subseteq L$ for some field $L\in\mathcal R$. Then the above upper bound for $h(f)$ implies
\begin{equation}\label{eq:hf}
h(f)\leq 7g\mu_X. 
\end{equation}
Explicit results of R\'emond \cite{remond:rational} bound the theta height $h_\theta(J)$  of the Jacobian $J$ of $X$ in terms of $h(f)$, see \cite[Section 4]{jore:shafarevich}. Pazuki \cite{pazuki:heights} worked out explicitly  a theorem of Bost-David that compares $h_F(J)$ to $h_\theta(J)$. We combine these results as in \cite[Section 4]{jore:shafarevich} and then (\ref{eq:hf}) leads to an estimate for $h_F(J)$ as stated in (i).

To prove (ii) we observe that our assumptions imply that $X$ is a hyperelliptic curve over $\QQ$ with a $\QQ$-rational Weierstrass point.  Let $\mathcal O$ be the ring of $S(\QQ,2)$-integers, let $\mathcal O^\times$ be the units of $\mathcal O$ and define the exponential height function $H=\exp(h)$. From \cite[Proposition 4.4 (i)]{rvk:hyperelliptic} we deduce that there is a Weierstrass model $\mathcal W(f)$ of $X$ over $\spe(\mathcal O)$ such that $f\in \mathcal O[x]$ is monic of degree $\nu=2g+1$ and has discriminant $\Delta(f)\in \mathcal O^\times$ that satisfies
\begin{equation}\label{eq:deltaw}
H(\Delta(f))\leq (2N_X)^{4g\nu}. 
\end{equation}
Since $\nu\geq 3$ and $\Delta(f)\neq 0$ we can choose distinct roots $\alpha,\beta,\gamma$ of $f$. We get that $\QQ(\alpha,\beta,\gamma)\subseteq L$ for some $L\in\mathcal R$.  The roots $\alpha,\beta,\gamma$ of our monic $f\in \mathcal O[x]$ are $S$-integral and $\Delta(f)$ is a $S$-unit, where $S=S(L,2)$. 
Thus we obtain
\begin{equation*}
\lambda=\frac{\gamma-\alpha}{\alpha-\beta}\in \mathcal O_S^\times \textnormal{ and } 1-\lambda\in \mathcal O_S^\times.
\end{equation*}
This implies  $h(\lambda)\leq \mu_S\leq\mu_X$ and  $H(\Delta(f)(\alpha-\beta)^{-\nu(\nu-1)})\leq(2\mu^2)^{\nu(\nu-1)}$, where $\mu=\exp(\mu_X)$. 
Therefore we deduce
\begin{equation}\label{eq:ab}
H(\alpha-\beta)\leq (2\mu^2)H(\Delta(f))^{1/(\nu(\nu-1))}.
\end{equation}
Since $f$ has coefficients in $\mathcal O$ we get that the trace $\textnormal{Tr}(f)$ of $f$ is in $\mathcal O$ and then \cite[Lemma 6]{evgy:binaryforms} gives $\eta\in \ZZ$, $\theta\in \mathcal O$ such that $\eta=\textnormal{Tr}(f)-\nu\theta$ and that $\left|\eta\right|\leq \nu$.
Thus for any root $\alpha$ of $f$ we get $\nu(\alpha-\theta)=\sum(\alpha-\beta)+\eta,$ with the sum taken over the roots $\beta$ of $f$. This together with (\ref{eq:ab}) leads to
$
H(\alpha-\theta)\leq \nu^3(2\mu^2)^{\nu-1}H(\Delta(f))^{1/\nu}.
$
We write $\theta^*f(x)=\prod(x-(\alpha-\theta))$ with the product taken over all roots $\alpha$ of $f$. The above estimate for $H(\alpha-\theta)$ combined with (\ref{eq:deltaw})  shows
\begin{equation}\label{eq:tfbound}
H(\theta^*f)\leq (\nu^32^{\nu})^\nu(2N_X)^{4g\nu}\mu^{2\nu(\nu-1)}. 
\end{equation}
Let $a_0$ be the smallest positive integer such that $a_0\theta^*f\in \ZZ[x]$. It satisfies $a_0\leq H(\theta^*f)$ 
and then we see that $\mathcal W=\mathcal W(a_0^2\theta^*f)$ has the desired properties.
This completes the proof of Proposition \ref{proph}.
\end{proof}

In the next section we shall use  Proposition \ref{proph} to bound the discriminants of some curves in terms of the quantity $\mu_X$.

\section{Discriminants}\label{secdiscriminants}

In this section we use Arakelov geometry for curves to relate discriminants to other invariants. In the first part we state and discuss Proposition \ref{propd} and \ref{propde}, which deal with curves of genus at least two and genus one curves respectively. In the second part we give the proofs.

Let $K$ be a number field of degree $d$ over $\QQ$ and let $g\geq 1$ be an integer. In the sequel $c_\delta$ denotes a constant depending only on $g$. Let $h_F$ be the Faltings height and let $\nu=2g+1$. We denote by $\delta(X_\CC)$ Faltings' delta invariant of a compact connected Riemann surface $X_\CC$ of genus at least one, defined in \cite[p.402]{faltings:arithmeticsurfaces}. If $X$ is a hyperelliptic curve over $K$ of genus $g\geq 2$, then $\Delta_X$ denotes its minimal discriminant and  $\mu_X$ the quantity from Section \ref{secheights}. We recall that any genus two curve over $K$ is a hyperelliptic curve over $K$.

\begin{proposition}\label{propd}
Suppose $X$ is a curve over $K$ of genus $g\geq 2$ with discriminant $D_X$, conductor $N_X$ and Jacobian $J$. Then the following statements hold.
\begin{itemize}
\item[(i)] If $X$ is a semi-stable hyperelliptic curve, then $\log D_X\leq 2^{2^{23}9^g}d\mu_X.$
\item[(ii)] Suppose $K=\QQ$ and $X$ is a hyperelliptic curve with a $\QQ$-rational Weierstrass point. Then it holds $\log \Delta_X\leq 8\nu^3(\mu_X+\log N_X+2\log(\nu)).$
\item[(iii)] If $X$ has genus two, then $D_X$ divides $\Delta_X.$
\item[(iv)] Suppose $X$ is semi-stable. Then $\log D_X\leq 12d(h_F(J)+c_\delta).$
\item[(v)] If $X_\CC$ is a compact connected Riemann surface of genus two, then $-186\leq \delta(X_\CC)$. 
\end{itemize}
\end{proposition}

The semi-stable assumptions in (i) and (iv) are used to apply results from Arakelov geometry. The constant $c_\delta$ in (iv) comes from lower bounds for $\delta(X_\CC)$ which are not known to be effective. However, the proof of (i) gives in addition an effective constant $c_{g}$, depending only on $g$, such that if $X$ is a semi-stable hyperelliptic curve over $K$ of genus $g\geq 2$, then $$\log D_X\leq 13d(h_F(J)+c_g).$$ Moreover, if $X$ is a semi-stable genus two curve over $K$, then (v) combined with the proof of (iv) implies that $\log D_X\leq 12d(h_F(J)+16)$. 

Let $X$ be a genus one curve over $K$. We may and do take  a finite extension $L$ of $K$ such that the base change $X_L$, of $X$ to $L$, is semi-stable. Then we define the unstable discriminant $\gamma_X$ of $X$ by $$\gamma_{X}=D_XD_{X_L}^{-1/l},$$ where $D_{X_L}$ is the discriminant of $X_L$ and $l$ is the relative degree of $L$ over $K$. We observe that the definition of $\gamma_X$ is independent of the choice of $L$.
The Jacobian $E$ of $X$ is an elliptic curve over $K$, and we denote by  $N_E$, $D_E$ and $\Delta_E$  the conductor, the discriminant and the minimal discriminant of  $E$ respectively. Let $T_0$ be the set of finite places $v$ of $K$ with $2\nmid N_v$ where $E$ has potential good reduction, and let $T_1$ be the set of finite places $v$ of $K$ with $2\nmid N_v$  where $E$ has reduction type $I_n^*$ for some integer $n\geq 1$. Here $I_n^*$ denotes the Kodaira symbol.  Let $\mu_E$ be the invariant of $E$ defined in the previous section, and for any finite place $v$ of $K$ let $n_v$ be the exponent of $\Delta_E$ at $v$ defined in Section \ref{secconjectures}.

We mention that statement (i) and (ii) in the next proposition follow directly from results of Liu-Lorenzini-Raynaud in \cite{lilora:genusone}.

\begin{proposition}\label{propde}
If $X$ is a genus one curve over $K$ with Jacobian $E$, then
\begin{itemize}
\item[(i)] $D_X=D_E,$ $N_X=N_E,$
\item[(ii)] $D_E=\Delta_E,$
\item[(iii)] $\log D_X\leq 6d\mu_E+\log\gamma_X+57d$ and
\item[(iv)]  $\log \gamma_X=\sum_{v\in T_0}n_v\log N_v+6\sum_{v'\in T_1}\log N_{v'}+r$ for $r$ a real number with $0\leq r\leq 24d\log 2$.
\end{itemize}
\end{proposition}

The Noether formula (\ref{eq:faltingsheight}) gives that the dependence on $\gamma_X$ and $\mu_E$ is optimal in (iii). On the other hand, the constants 57 and 24, in (iii) and (iv) respectively, can be slightly improved.

In the remaining part of this section we prove the above results. The tools used in the proof of Proposition \ref{propde} are as follows. Liu-Lorenzini-Raynaud \cite{lilora:genusone} give (i) and (ii). We also mention that (ii) is implicitly contained in Saito's proof \cite{saito:conductor} of Ogg's formula. On using an explicit result of Faltings \cite{faltings:arithmeticsurfaces} we shall derive (iii) from Proposition \ref{prophe}. Finally, (iv) uses the classification of Kodaira-N\'eron \cite{neron:models}.

\begin{proof}[Proof of Proposition \ref{propde}]
Let $X$ be a genus one curve over $K$, with Jacobian $E$. If $v$ is a finite place of $K$, then $n_v,a_v,\delta_v$ are as in Section \ref{secconjectures}.

We start with the proof of (i). The global result $D_X=D_E$ is a direct consequence of the local equality $\delta_v(X)=\delta_v(E)$ in  \cite[Theorem 5.9]{lilora:genusone}. This local equality  implies in addition that the curve $X$ has good reduction at $v$ if and only if its Jacobian $E$ has good reduction at $v$. This gives $a_v(X)=0$ and we deduce that $N_X=N_E$.

Part (ii) follows from $n_v(E)=\delta_v(E)$ in \cite[p.509]{lilora:genusone}. We now prove (iii). For any embedding $\sigma:K\hookrightarrow \mathbb C$ we may and do take $\tau_\sigma\in \mathbb C$ such that the base change of $E$ to $\mathbb C$ with respect to $\sigma$ takes the form $\mathbb C/\ZZ+\ZZ\tau_\sigma$ and such that $\im(\tau_\sigma)\geq \sqrt{3}/2$. The height $h_F(E)$ is invariant under an algebraic base change of $E$. Hence \cite[Theorem 7 b)]{faltings:arithmeticsurfaces}, which holds for all genus one curves over $K$, leads to
\begin{equation}\label{eq:faltingsheight}
\log D_X=12dh_F(E)+\log \gamma_X+\sum\log\left|(2\pi)^{12}\Delta(\tau_\sigma)\im(\tau_\sigma)^6\right|,
\end{equation}
with the sum taken over all embeddings $\sigma:K\hookrightarrow \mathbb C$, where
$\Delta(\tau_\sigma)=q\prod_{n=1}^{\infty}(1-q^n)^{24}, \ \  q=\exp(2\pi i \tau_\sigma).$ On using that $\log \left|\Delta(\tau_\sigma)/q\right|\leq 24\left|q\right|/(1-\left|q\right|)$ and $\left|q\right|\leq \exp(-\pi\sqrt{3})$, we compute
$\log \left|(2\pi)^{12}\Delta(\tau_\sigma)\im(\tau_\sigma)^6\right|\leq 16$.
Thus (\ref{eq:faltingsheight}) and (\ref{eq:hs}) imply assertion (iii). 

To show (iv) we use (i) and (ii). They give that the unstable discriminant $\gamma_E$ of $E$ coincides with $\gamma_X$. Then the classification \cite[p.448]{silverman:aoes} and a calculation for finite places $v$ of $K$ with $2\vert N_v$ imply (iv). Thus all statements of Proposition \ref{propde} are proven.
\end{proof}

In the proof of Proposition \ref{propd} we shall use two lemmas which describe properties of a Siegel modular form associated to a hyperelliptic Riemann surface. To state and prove these lemmas we now introduce some notation.  Let $\mathfrak H_g$ be the Siegel upper half plane of complex symmetric $g\times g$ matrices with positive definite imaginary part.  For any $\tau\in \mathfrak H_g$ and $z\in \mathbb C^g$ we write
$$
\vartheta_{(a,b)}(\tau,z)=\sum_{m\in \ZZ^g}\exp(i\pi (m+a)^*\tau (m+a)+2\pi i (m+a)^*(z+b)),
$$
where $a,b\in \mathbb R^{g}$ and $u^*$ denotes the dual of $u\in \mathbb R^g$. If $\mathcal S$ is a subset of $\mathcal I=\{1,2,\dotsc,2g+1\}$, then we define the theta characteristic $\eta_{\mathcal S}\in \frac{1}{2}\mathbb Z^{2g}$ as in \cite[Chapter IIIa]{mumford:theta2}. Let $n=\binom{2g}{g+1}$ and write
$$\Delta_g(\tau)=2^{-4n(g+1)}\prod\vartheta_{\eta_{\mathcal S\circ \, \mathcal U}}(\tau,0)^8$$
with the product taken over all subsets $\mathcal S$ of $\mathcal I$ of cardinality $g+1$, where $\circ$ is the symmetric difference operator and $\mathcal U=\{1,3,\dotsc,2g+1\}$. If $g=1$, then $\Delta_1(\tau)$ is the classical discriminant modular form. The group $\textnormal{Sp}_{2g}(\ZZ)$ of $2g\times 2g$ symplectic matrices with coefficients in $\ZZ$ acts on $\mathfrak H_g$ and $\frac{1}{2}\ZZ^{2g}$ as follows. If $\sigma=(\alpha,\beta;\gamma,\delta)\in \textnormal{Sp}_{2g}(\ZZ)$, $\tau\in \mathfrak H_g$ and $u\in \frac{1}{2}\ZZ^{2g}$, then
$$\sigma \tau=(\alpha\tau + \beta)(\gamma\tau+\delta)^{-1}, \ \ \sigma u= \begin{pmatrix} \delta & -\gamma\\ -\beta & \alpha \end{pmatrix} u+\frac{1}{2}\begin{pmatrix} \textnormal{diag}(\gamma\delta^t) \\ \textnormal{diag}(\alpha\beta^t) \end{pmatrix}, $$
where $\textnormal{diag}(M)$ denotes the vector of diagonal elements of a matrix $M$ and $M^t$ denotes the transpose of $M$. Let $r=\binom{2g+1}{g+1}$ and write $$\sigma \Delta_g(\tau)=2^{-4n(g+1)}\prod\vartheta_{\sigma^{-1}\eta_{\mathcal S\circ  \, \mathcal U}}(\tau,0)^8, \ \ \sigma\in \textnormal{Sp}_{2g}(\ZZ)$$ with the product taken over all subsets $\mathcal S$ of $\mathcal I$ with cardinality $g+1$. Let $\det(M)$ be the determinant of $M$. We shall use the following completely elementary lemma.
\begin{lemma}\label{lemdiscestimate}
If $\sigma\in\textnormal{Sp}_{2g}(\ZZ)$ and $\tau\in \mathfrak H_g$, then $$\left|\Delta_g(\sigma \tau)\right|\det(\im(\sigma \tau))^{2r}=\left|\sigma \Delta_g(\tau)\right|\det(\im (\tau))^{2r}.$$
\end{lemma}
\begin{proof}
Let $\sigma=(\alpha,\beta;\gamma,\delta)\in\textnormal{Sp}_{2g}(\ZZ)$, $\tau\in \mathfrak H_g$ and $u\in \frac{1}{2}\ZZ^{2g}$. It holds $\det(\im(\sigma \tau))=\det(\im(\tau))\left|\det(\gamma\tau+\delta)\right|^{-2}$ and transformation properties of theta functions in \cite[Chapter V]{igusa:theta} give $\vartheta_{\sigma u}(\sigma\tau,0)^8=\det(\gamma\tau+\delta)^{4}\vartheta_{u}(\tau,0)^8$. Therefore we deduce
$$\left|\vartheta_{u}(\sigma\tau,0)\right|\det(\im(\sigma \tau))^{1/4}=\left|\vartheta_{\sigma^{-1} u}(\tau,0)\right|\det(\im(\tau))^{1/4}$$
and then on writing out the definitions of $\Delta_g(\sigma\tau)$ and $\sigma \Delta_g(\tau)$ we derive Lemma \ref{lemdiscestimate}.
\end{proof}

We denote by $\mathfrak F_g$ the standard fundamental domain of Siegel for the action of $\textnormal{Sp}_{2g}(\ZZ)$ on $\mathfrak H_g$. In the sequel $k_1$ denotes an effective constant depending only on $g$. It is given explicitly in (\ref{eq:estfin}). We obtain the following lemma. 
\begin{lemma}\label{lemfg}
Any $\tau \in \mathfrak H_g$ satisfies $\left|\Delta_g(\tau)\right|\det(\im(\tau))^{2r}\leq k_1.$
\end{lemma}
This lemma may be of independent interest. For example, it will be used in \cite{javk:smallpoints} to establish the hyperelliptic case of Szpiro's small points conjecture. We also mention that if $\tau\in \mathfrak F_g$, then the proof of Lemma \ref{lemfg} shows in addition  that  $\left|\Delta_g(\tau)\right|\leq k_1$.
\begin{proof}[Proof of Lemma \ref{lemfg}]
In the first part of the proof we assume that $\tau\in\mathfrak F_g$ to bound the absolute value of certain theta functions evaluated in $\tau$. In the second part we use Lemma \ref{lemdiscestimate} to establish the desired estimate for all $\tau\in \mathfrak H_g$.

We take $a=(a_i),b\in \mathbb R^g$ and we assume that $\tau=(\tau_{ij})\in\mathfrak F_g$. Thus $\im (\tau_{ii})\geq \sqrt{3}/2$ and any $m=(m_i)\in \mathbb R^g$ satisfies the Minkowski inequality $$m^*\im(\tau) m\geq k_2\sum_{i=1}^{g}\im(\tau_{ii})m_i^2, \ k_2=(4/g^3)^{g-1}(3/4)^{g(g-1)/2}.$$
Then the definition of $\vartheta_{(a,b)}$ together with elementary analysis implies that $\left|\vartheta_{(a,b)}(\tau,0)\right|$ is at most $ \prod_{i=1}^g\sum_{m\in \ZZ}\exp(-\pi k_2\im (\tau_{ii})(m+a_i)^2)$. 
This leads to
\begin{equation}\label{eq:thetaf}
\left|\vartheta_{(a,b)}(\tau,0)\right|\leq (k_2^{-1}+2)^{g}. 
\end{equation}
Hadamard's theorem gives that $\det(\im(\tau))\leq \prod_{i=1}^g\im(\tau_{ii})$ which implies that $\det(\im(\tau))\leq\im(\tau_{gg})^g$ for our $\tau\in\mathfrak F_g$. Thus the above displayed Minkowski inequality shows 
$\left|\vartheta_{(a,b)}(\tau,0)\right|\det(\im(\tau))^{2r}\leq A_1A_2,$ where
$$A_1=\prod_{i=1}^{g-1}\sum_{m\in \ZZ}\exp(-\pi k_2 \im(\tau_{ii}) (m+a_i)^2) \textnormal{ and }$$
$$\ \ \ \ A_2=\sum_{m\in \ZZ}\exp(-\pi k_2 \im(\tau_{gg}) (m+a_g)^2) \im(\tau_{gg})^{2gr}.$$
If $a\in \frac{1}{2}\ZZ^g$, $a_g\notin\ZZ$, then any $m\in\ZZ$ satisfies $(m+a_g)^2\geq 1/4$, and we deduce an estimate for $A_2$ which together with the above arguments leads to
\begin{equation}\label{eq:im}
\left|\vartheta_{(a,b)}(\tau,0)\det(\im(\tau))^{2r}\right|\leq (6/k_2+2)^{g}(2gr/k_2)^{2gr}. 
\end{equation}
Next, we observe that any $\tau\in \mathfrak H_g$ takes the form $\tau=\sigma \bar{\tau}$ with $\sigma\in \textnormal{Sp}_{2g}(\ZZ)$ and $\bar{\tau}\in\mathfrak F_g$. An application of Lemma \ref{lemdiscestimate} gives
\begin{equation}\label{eq:app}
\left|\Delta_g(\tau)\right|\det(\im(\tau))^{2r}=\left|\sigma\Delta_g(\bar{\tau})\right|\det(\im (\bar{\tau}))^{2r}.
\end{equation}
From \cite[Lemma 3.4]{lockhart:discriminant} we get a subset $\mathcal S'$ of $\mathcal I$ of cardinality $g+1$ such that the characteristic $\sigma^{-1} \eta_{\mathcal S'\circ  \, \mathcal U}=(a,b)$ with $(a,b)\in \frac{1}{2}\ZZ^{2g}$ and $a_g\notin \ZZ.$ By definition, the right hand side of (\ref{eq:app}) takes the form
$$2^{-4n(g+1)}\left(\prod\left|\vartheta_{\sigma^{-1}\eta_{\mathcal S\circ  \, \mathcal U}}(\bar{\tau},0)\right|^8\right)\left|\vartheta_{(a,b)}(\bar{\tau},0)\right|^{8}\det(\im(\bar{\tau}))^{2r}$$ with the product taken over all subsets $\mathcal S\neq \mathcal S'$ of $\mathcal I$ with cardinality $g+1$. Then  from  (\ref{eq:thetaf}), (\ref{eq:im}) and (\ref{eq:app}) we deduce that the constant
\begin{equation}\label{eq:estfin}
k_1=2^{-4n(g+1)}(6/k_2+2)^{8gr}(2gr/k_2)^{2gr} 
\end{equation}
has the desired property. This concludes the proof of the lemma.
\end{proof}

We remark that the proof of Lemma \ref{lemfg} shows in addition that one can replace the exponent $2r$ by any positive real number. The resulting estimate is still effective and depends only on $g$ and this number.

To describe the main ingredients for the proof of Proposition \ref{propd} we let $D_X$ and $h_F(J)$ be as in the statement. The proof of (i) consists of two parts. In the first part we  derive an effective upper bound for $D_X$ in terms of $h_F(J)$.\footnote[2]{We are grateful to Robin de Jong for his precise answer to our question in which he proposed a first strategy to upper bound $D_X$ effectively in terms of $h_F(J)$.} Here we combine explicit results from Arakelov geometry of de Jong \cite{dejong:weierstrasspoints} with Lemma \ref{lemfg}. In the second part  we use the estimate for $h_F(J)$ in Proposition \ref{proph} (i). Assertion (ii) is a direct consequence of Proposition \ref{proph} (ii), and (iii) follows from results of Liu \cite{liu:conductor} and of Saito \cite{saito:genus2} for discriminants of genus two curves. We deduce (iv)  on combining the Noether formula with other results for arithmetic surfaces of Faltings \cite{faltings:arithmeticsurfaces}. To prove (v) we combine Lemma \ref{lemfg} with Bost's formula \cite[Proposition 4.1]{bost:genus2} for Faltings' delta invariant of compact connected genus two Riemann surfaces.

\begin{proof}[Proof of Proposition \ref{propd}]
We begin with the proof of (i) and we take a semi-stable curve $X$ over $K$ of genus $g\geq 2$. For any embedding $\sigma:K\hookrightarrow \mathbb C$ we denote by $X_\sigma$ the compact Riemann surface corresponding to the base change of $X$ to $\mathbb C$ with respect to $\sigma$. Let $T(X_\sigma)$ be the invariant of de Jong. It is the norm of a canonical isomorphism between certain line bundles on $X_\sigma$ and we refer to \cite[Definition 4.2]{dejong:riemanninvariants} for a precise definition. Let $B$ be the spectrum of the ring of integers of $K$ and let $p:\mathcal X\to B$ be the minimal regular model of $X$ over $B$. For a closed point $v$ of $B$ we let $\delta_v$ be the discriminant exponent of $X$ at $v$ defined in Section \ref{secconjectures}. Our semi-stable assumption gives that $\delta_{v}$  coincides  with the number of singular points of the geometric fiber over $v$ of $p:\mathcal X\to B$. Hence \cite[Corollary 4.4]{dejong:weierstrasspoints} implies
\begin{equation}\label{eq:dxest}
\log D_X\leq d(13h_F(J)+10g)-4\sum\log T(X_\sigma)
\end{equation}
with the sum taken over all embeddings $\sigma: K\hookrightarrow \mathbb C$.
To estimate $T(X_\sigma)$ in terms of $h_F(J)$ we combine Lemma \ref{lemfg} with the explicit description of $T(X_\sigma)$ of de Jong. Our assumption in (i), that $X$ is hyperelliptic, gives a finite morphism $\varphi:X_\sigma\to \mathbb P_{\mathbb C}^1$ of degree two, for $\mathbb P_{\mathbb C}^1$ the projective line over $\mathbb C$. Let $H_1(X_\sigma,\ZZ)$ be the first homology group of $X_\sigma$ with coefficients in $\ZZ$. On following Mumford \cite[Chapter IIIa]{mumford:theta2} we construct a canonical symplectic basis of $H_1(X_\sigma,\ZZ)$ with respect to a fixed ordering of the $2g+2$ branch points of $\varphi$. Then \cite[Theorem 4.7]{dejong:riemanninvariants} provides a basis of the global sections of the sheaf of differentials on $X_\sigma$ with the following property: Integration of this basis over the canonical symplectic basis of $H_1(X_\sigma,\ZZ)$ gives a period matrix $\tau_\sigma \in\mathfrak H_g$ that satisfies
\begin{equation*}
T(X_\sigma)=(2\pi)^{-2g}\left|\Delta_g(\tau_\sigma)\det(\im (\tau_\sigma))^{2r}\right|^{-(3g-1)/(8ng)}.
\end{equation*}
Thus Lemma \ref{lemfg} and (\ref{eq:estfin}) show $-\log T(X_\sigma)\leq 36g^3$ and then Proposition \ref{proph} (i) together with (\ref{eq:dxest}) implies (i).

To show (ii) we may and do assume that $X$ is a hyperelliptic curve over $\QQ$ of genus $g\geq 2$ with a $\QQ$-rational Weierstrass point. We go into the proof of Proposition \ref{proph} (ii) and we continue the notation introduced there. From \cite[Equation (4.1)]{rvk:hyperelliptic} we get that the discriminant $\Delta$ of the Weierstrass model $\mathcal W(a_0^2\theta^*f)$ of $X$ over $\spe(\ZZ)$ satisfies $\Delta=\Delta(f)a_0^{4\nu}$ with $\nu=2g+1.$
Then (\ref{eq:deltaw}) combined with $a_0\leq H(\theta^*f)$ and (\ref{eq:tfbound}) lead to
$$\log \left|\Delta\right|\leq 16\nu^3 \log \nu+8\nu^3(\mu_X+\log N_X).$$
For a finite place $v$ of $\QQ$ we let $n_v$ be the exponent of the minimal discriminant $\Delta_X$ of $X$ defined in Section \ref{secconjectures}.  The base change of $\mathcal W(a_0^2\theta^*f)$ to the spectrum of the local ring $\mathcal O_{\QQ,v}$ at $v$ is a Weierstrass model of $X$ with discriminant $\Delta$ and then the minimality of $n_v$ provides $n_v\leq v(\Delta)$. We conclude
$\Delta_X\leq \left|\Delta\right|$ which together with the above estimate for $\log \left|\Delta\right|$ implies (ii).

To prove (iii) we take a genus two curve $X$ over $K$. In the first part we follow Saito \cite[p.229-230]{saito:genus2}. Let $v$ be a finite place of $K$ and let $X_v$ be the base change of $X$ to the completion $K_v$ of $K$ at $v$, with structural morphism $\rho_\eta:X_v\to \spe(K_v)$.
We continue the notation of Section \ref{secconjectures}. Further, we say that a morphism is canonical if it commutes with arbitrary base change. There exists a canonical isomorphism
$$S^2({\rho_\eta}_*(\omega_{X_v/K_v}))\to {\rho_\eta}_*(\omega_{X_v/K_v}^{\otimes 2}),$$
for $S^2$ the second symmetric power, which induces a canonical isomorphism
$$D'': (\det R{\rho_\eta}_*(\omega_{X_v/K_v}))^{\otimes 3}\to \det R{\rho_\eta}_*(\omega_{X_v/K_v}^{\otimes 2}).$$
The Siegel  modular form of weight 10,  given by the canonical isomorphism
$$D': \mathcal O_{\spe(K_v)}\to (\det R{\rho_\eta}_*(\omega_{X_v/K_v}))^{\otimes 10},$$
takes the form $D'=(DD'')\otimes \textnormal{id}$ for $D$ the canonical isomorphism from Section \ref{secconjectures}. We define with $D'$ and $D''$  integers $\delta'_v$ and ${\delta}_v ''$ respectively in the same way as we defined in Section \ref{secconjectures} with $D$ the integer $\delta_v$. The integers $\delta'_v$ and ${\delta}_v ''$ are both non-negative and the above composition of $D'$ shows
\begin{equation}\label{eq:ord}
\delta'_v=\delta_v+{\delta}_v ''.
\end{equation}
To relate the number $\delta'_v$ to the exponent $n_v$ of the minimal discriminant $\Delta_X$ of $X$ we may and do pass to the strict henselization of the valuation ring of $K_v$. Now, on combining \cite[Corollaire]{liu:conductor} and \cite[Proposition 2 (d)]{liu:conductor}, we see that $n_v\geq \delta'_v$ and then (\ref{eq:ord}) together with ${\delta}_v ''\geq 0$ implies  $n_v\geq \delta_v$. We conclude (iii).

Next, we verify (iv). By assumption, $X$ is a semi-stable curve over $K$ of genus $g\geq 2$.  Let $\omega_{\mathcal X/B}$ be the relative dualizing sheaf of the minimal regular model $p:\mathcal X\to B$ of $X$ over $B$, for $B$ the spectrum of the ring of integers of $K$. We write $\omega^2$ for the Arakelov self-intersection of $(\omega_{\mathcal X/B},\lVert\cdot\rVert)$, where $\lVert\cdot\rVert$ is the Arakelov metric defined in \cite[p.392]{faltings:arithmeticsurfaces}.  Our semi-stable assumption combined with the Noether formula \cite[Theorem 6]{faltings:arithmeticsurfaces} 
gives
\begin{equation*}
12dh_F(J)=\log D_X+\omega^2+\sum \delta'(X_\sigma) 
\end{equation*}
with the sum taken over all embeddings $\sigma:K\hookrightarrow \mathbb C$. Here  $\delta'(X_\sigma)$ is (up to an absolute constant) Faltings' delta invariant of the Riemann surface associated to the base change of $X$ to $\mathbb C$ with respect to $\sigma$.  Let $M_g(\CC)$ denote the moduli space of smooth projective connected complex curves of genus $g$. Faltings' delta invariant, viewed as a function $M_g(\CC)\to \mathbb R$, has a minimum. Hence, there is a constant $c_\delta$, which depends at most on $g$, such that $-\delta'(X_\sigma)\leq c_\delta$. Since $\omega^2\geq 0$  by \cite[Theorem 5]{faltings:arithmeticsurfaces}, the displayed formula implies (iv). 

Finally, we prove (v). Let $X_\CC$ be a compact connected Riemann surface of genus two and let $\textnormal{Pic}_1(X_\CC)$ be the component of degree one of the Picard group $\textnormal{Pic}(X_\CC)$ of $X_\CC$. We define the function $\lVert\theta\rVert$  on $\textnormal{Pic}_1(X_\CC)$ and  the translation invariant differential form $\mu$ on $\textnormal{Pic}(X_\CC)$ as in \cite[Section 1]{bost:genus2}. We define $$\log\lVert H\rVert(X_\CC)=\int_{\textnormal{Pic}_1(X_\CC)}\log \lVert\theta\rVert \frac{\mu^2}{2}.$$ 
Let $\lVert \Delta_2\rVert(X_\CC)$ be the real number defined in \cite[Section 3.3]{bost:genus2}.
Bost's formula \cite[Proposition 4.1]{bost:genus2} gives that Faltings' delta invariant $\delta(X_\CC)$ of $X_\CC$ satisfies 
$$\delta(X_\CC)=-16\log (2\pi)-\log \lVert \Delta_2\rVert(X_\CC)-4\log\lVert H\rVert(X_\CC).$$
Let $\textnormal{Pic}_0(X_\CC)$ be the component of $\textnormal{Pic}(X_\CC)$ of degree zero. We take $\tau\in \mathfrak H_2$ with $\textnormal{Pic}_0(X_\CC)\cong \CC^2/\ZZ^2+\tau \ZZ^2$. Then it follows
$$\lVert \Delta_2\rVert(X_\CC)^4=\det(\im(\tau))^{2r}\lvert \Delta_2(\tau)\rvert$$  for $r=\binom{5}{3}$. 
Further, we see $\log\lVert H\rVert(X_\CC)\leq 1/2$, since 
$\int\lVert\theta\rVert^2\mu^2/2=1/2$ with the integral taken over $\textnormal{Pic}_1(X_\CC)$. Thus Lemma \ref{lemfg} and the displayed formulas imply (v). This completes the proof of Proposition \ref{propd}.
\end{proof}

\section{Unit equations}\label{secunits}

Let $K$ be a number field and let $T$ be a finite set of places of $K$. In this section we apply the theory of logarithmic forms to bound the height of solutions to $T$-unit equations in a finite extension $L$ of $K$. 

We start with some notation. Let $d$ be the degree of $K$ over $\QQ$ and let  $l$ be the relative degree of $L$ over $K$. We denote by $D_K$ and $D_L$ the absolute value of the discriminant of $K$ over $\QQ$ and of $L$ over $\QQ$ respectively. Let $U$ be the places of $L$ which divide a finite place in $T$ and let $\mathcal O_U^\times$ be the $U$-units in $L$. As above we write $h(\lambda)$ for the usual absolute logarithmic Weil height of $\lambda\in L$ and we define
$$\mu_{U}=\sup(h(\lambda), \lambda \in \mathcal O_U^\times \textnormal{ and } 1-\lambda \in \mathcal O_U^\times).$$
Let $S$ be a finite set of places of $K$ such that $L$ is unramified outside $S$. Let $t$ be the cardinality of the finite places in $T$, let $N_T=\prod N_v$ with the product taken over the finite places $v$ in $T$ and let $s$, $N_S$ be the corresponding quantities of $S$. Finally, we take $n=ld$ and $m=\max(6,l)$.

\begin{proposition}\label{prope}
Suppose $r,\epsilon>1$ are real numbers. Then the following statements hold.
\begin{itemize}
\item[(i)] There exists an effective constant $k_3$, depending only on $r,l,d$, such that if $S\subseteq T$, then it holds $\mu_{U}\leq k_3N_T^{3l/2} (N_T^{2n}D_K^{l/2})^r.$
\item[(ii)] Suppose $S\subseteq T$. Then $\mu_{U}\leq (2m dN_T^{\log m})^{15m d-1}D_K^{m-1}.$
\item[(iii)] If $(abc)$ is true for $n,r,\epsilon$ with the constant $c$, then $\mu_{U}\leq  r\log N_T +\frac{\epsilon}{n}\log (N_S^{l-1}l^{sn})+\frac{\epsilon}{d}\log D_K+\frac{1}{n}\log c.$
\end{itemize}
\end{proposition}

Part (i) and (ii) shall follow from Gy\H{o}ry-Yu \cite[Theorem 1]{gyyu:sunits}. Their theorem is based on the theory of logarithmic forms and we refer to Baker-W\"ustholz \cite{bawu:logarithmicforms} in which the state of the art of this theory is exposed.

To prove Proposition \ref{prope} we shall use inter alia two lemmas. The first lemma is a direct consequence of the classical Dedekind discriminant theorem.

\begin{lemma}\label{lem:dl}
The discriminant $D_L$ is at most $D_K^lN_S^{l-1}l^{ns}.$
\end{lemma}

\begin{proof}
Let $\mathfrak d_{L/K}$ be the relative discriminant ideal of $L$ over $K$, let $v$ be a finite place of $K$ of residue characteristic $q$ and let $(w)$ be the places of $L$ which divide $v$. If  $v$ is not in $S$, then Dedekind's discriminant theorem gives $$v(\mathfrak d_{L/K})=0.$$
We now suppose that $v\in S$. Let $e_v$ and $\mathfrak f_v$ be the ramification index  and the  residue degree of $v$ over $q$ respectively and let $(e_{w/v})$ and $(\mathfrak f_{w/v})$ be the ramification indexes and residue degrees of $(w)$ over $v$ respectively.  Dedekind's discriminant theorem \cite[B.2.12]{bogu:diophantinegeometry} delivers
$$
v(\mathfrak d_{L/K})\leq\sum \mathfrak f_{w/v}(e_{w/v}-1+w(e_{w/v})) 
$$
with the sum taken over all $w\in (w)$. Hence $w(e_{w/v})=e_{w/v}e_v\ord_q(e_{w/v})$ together with $\ord_q(e_{w/v})\log N_v\leq  \mathfrak f_v \log e_{w/v}$ implies 
$$v(\mathfrak d_{L/K})\leq l-1+n\log l/\log N_v.$$
Let $\N$ be the norm from $K$ to $\QQ$. On combining $D_L=\N(\mathfrak d_{L/K})D_K^l$ with the displayed inequalities involving  $v(\mathfrak d_{L/K})$, we deduce Lemma \ref{lem:dl}.
\end{proof}

The second lemma is based on a completely explicit version of the prime number theorem due to Rosser-Schoenfeld \cite{rosc:formulas}. We define the quantity $n_\Sb=\prod \log N_v$ with the product taken over the finite places $v\in \Sb$.

\begin{lemma}\label{lem:analytic}If $\varepsilon>0$ is a real number, then there are effective constants $k_4,k_5,k_6$, depending only on $\varepsilon$ and $d$, such that $$t\leq \varepsilon \log N_T+k_4, \ \ t^t\leq k_5 N_T^{d+\varepsilon} \textnormal{ and }  n_T\leq k_6 N_T^{\varepsilon}.$$
\end{lemma}
\begin{proof} Let $\varepsilon> 0$ be a real number. We may and do assume that $t\geq 1$. 
Let $\omega$ be the number of rational prime divisors of the radical $\textnormal{rad}_T$ of $N_T$, and let $\vartheta (q)$ be the logarithm of the product taken over all rational primes at most the $\omega$-th rational prime $q$. The explicit estimates \cite[Theorem 4]{rosc:formulas} and \cite[Corollary of Theorem 3]{rosc:formulas} give that $q<(1+\varepsilon)\vartheta (q)+\log k_7$ and that $\omega\log \omega<q$ respectively, where $k_7$ is an effective constant which depends only on $\varepsilon$. 
Thus 
$
\omega\log \omega < (1+\varepsilon)\log \textnormal{rad}_T+\log k_7
$
and then $t\leq d\omega$ leads to 
\begin{equation}\label{eq:sharp}
t\leq \frac{d(1+\varepsilon)}{\log\omega}\log N_T+\frac{d}{\log \omega}\log k_7, \ \ \ t^t\leq N_T^{d(1+\varepsilon)}k_7^{d}d^t. 
\end{equation}
Hence, we see that there exist effective constants  $k_4$ and $k_5$ with the desired properties. 
To estimate $n_T$ we observe that any real number $r>1$ satisfies $\log r\leq \varepsilon^{-1}r^{\varepsilon}$. 
This implies that $n_T\leq (2/\varepsilon)^{t}N_T^{\varepsilon/2}$. From (\ref{eq:sharp}) we see that there is an effective constant $k_6$, depending only on $\varepsilon$ and $d$, such that $(2/\varepsilon)^{t}\leq k_6 N_T^{\varepsilon/2}$ and then we conclude Lemma \ref{lem:analytic}. 
\end{proof}

We now use the above lemmas to prove Proposition \ref{prope}.  The main tool used in the proof of (i) and (ii) is a result of Gy\H{o}ry-Yu \cite[Theorem 1]{gyyu:sunits}.

\begin{proof}[Proof of Proposition \ref{prope}]
We suppose that $\lambda$ and $1-\lambda$ are $\ls$-units in $L$.

To prove (i)  we take a real number $r>1$ and we assume that $S\subseteq T$. Let $R_U$ be the $U$-regulator of $U$ in $L$ and define $n_U$ in the same way as $n_\Sb$ with $\ls$ in place of $\Sb$.  We obtain a $U$-unit equation $$\lambda+(1-\lambda)=1$$ and thus  \cite[Theorem 1]{gyyu:sunits} gives an explicit constant $\kappa_T$, depending only on $l,t,d$, with
$$
h(\lambda)\leq \kappa_TN_\Sb^lR_{\ls}\max(1,\log R_{\ls}).
$$
It holds that $n_U\leq (l^{\ts}n_\Sb)^{l}$ and then \cite[Inequality (4.3)]{rvk:hyperelliptic}, with $L$ and $\ls$ in place of $K$ and $T$ respectively, shows
$$
R_{\ls}\leq(2n)^{n-1}D_L^{1/2}\max(1,D_L)^{n-1}(l^{\ts}n_\Sb)^{l}.
$$
Therefore  Lemma \ref{lem:dl}, our assumption $S\subseteq T$ and the above estimate for $h(\lambda)$ lead to
$\mu_{U}\leq \kappa_TN_T^lc_Kc_T\log (c_Kc_T),$
\begin{equation}\label{eq:precisebound}
\begin{split}
\kappa_T&\leq 2^{34}(ln(t+d))^{2l(t+d)+5}2^{7l(t+d)}, \\
c_K&=D_K^{l/2}(3m^3d^2\max(1,\log D_K))^{n-1},\\
c_\Sb&=(N_\Sb^{1/2} n_\Sb)^{l}(\max(1,\ts)m^{2\ts}\max(1, \log N_\Sb))^{md-1}.
\end{split}
\end{equation}
Then Lemma \ref{lem:analytic} implies (i).

To show (ii) we assume again $S\subseteq T$.
We calculate the effective constants in Lemma \ref{lem:analytic} for $\varepsilon=1/2$ and then (\ref{eq:precisebound}) leads to (ii). 

To prove (iii) we take real numbers $r,\epsilon>1$ and we assume that $(abc)$ holds for $n=ld,r,\epsilon$ with the constant $c$. The numbers $\alpha=\lambda$, $\beta=1-\lambda$ and $\gamma=1$ are in $L$ and satisfy $\alpha+\beta=\gamma$. Thus $(abc)$ gives $$nh(\lambda)\leq r\log S_L(\alpha,\beta,\gamma)+\epsilon\log D_L+\log c.$$
Hence, Lemma \ref{lem:dl} and $S_L(\alpha,\beta,\gamma)\leq N_T^{n}$ imply (iii), and this completes the proof of Proposition \ref{prope}.
\end{proof}

With these results we are now ready to prove our theorems and corollaries and this will be done in the next section.

\section{Proofs}\label{secproofs}

In this section we prove the results stated in Section \ref{secstatement}. We continue the notation introduced in the previous sections. For a number field $K$ we write $d$ for its degree over $\QQ$, $D_K$ for the absolute value of its discriminant over $\QQ$ and $h_K$ for the class number of its ring of integers $\OK$. As above we denote by $h(f)$ the absolute logarithmic Weil height of a polynomial $f$ with coefficients in a number field.

The following proof of Theorem \ref{thm1} is based on a reinterpretation of the effective Shafarevich conjecture in \cite{rvk:hyperelliptic} in terms of bad reduction.

\begin{proof}[Proof of Theorem \ref{thm1}]
We take a hyperelliptic curve $C$ over $K$ of genus $g\geq 1$ with conductor $N_C$ and minimal discriminant $\Delta_C$ as in the theorem. 
To reinterpret the main result of \cite{rvk:hyperelliptic} we write $\nn=6(2g+1)(2g)(2g-1)d^2$ and we let $S$ be the finite places of $K$ where $C$ has bad reduction. 
We denote by $s$ the cardinality of $S$, and we write $N_S=\prod N_v$ with the product taken over the places $v$ in $S$. The quantity $M_C$ from Section \ref{secconjectures} satisfies $N_S=M_C$ and $C$ has good reduction outside $S$. Thus \cite[Theorem]{rvk:hyperelliptic} gives a Weierstrass model $\mathcal W$ of $C$ over $\spe(\OK)$ with discriminant $\Delta$ such that
$$h(\Delta) \leq (\nn(s+h_K))^{5\nn(s+h_K)}M_C^{\nn/2}D_K^{\nn h_K/4}$$
if $C$ has a $K$-rational Weierstrass point and
$$h(\Delta) \leq (\nn(s+h_K))^{k_8(2\nn)^3(s+h_K)^4}M_C^{(3\nn)^3(s+h_K)^4}D_K^{(3\nn)^3(s+h_K)^4}$$
otherwise, where $k_8$ is an absolute effective constant. If $v$ is a finite place of $K$, then the base change of $\mathcal W$ to the spectrum of the local ring (in $K$) at $v$ is a Weierstrass model of $C$ with discriminant $\Delta$. Thus the minimality of the exponent $n_v$ defined in Section \ref{secconjectures} implies $n_v\leq v(\Delta)$. We conclude that
$\Delta_C\leq \N(\Delta)$ for $\N$ the norm from $K$ to $\QQ$ and then we deduce $$\Delta_C\leq \N(\Delta)=\prod \lvert\sigma(\Delta)\rvert\leq \prod\max(1,\lvert\sigma(\Delta)\rvert)=\exp(dh(\Delta))$$
with both products taken over all embeddings $\sigma:K\hookrightarrow \CC$.
We recall that $M_C\leq N_C$ and then  Lemma \ref{lem:analytic} combined with the above estimates for $h(\Delta)$ leads to upper bounds for $\Delta_C$ as stated in Theorem \ref{thm1}. To simplify the form of the constant $c_2$ in statement (ii) we used that $h_K\leq  5(4d)^{d}D_K^{3/2}$, which follows from standard inequalities.
This completes the proof of Theorem \ref{thm1}.
\end{proof}

Let $X$ be a curve over $K$ of genus $g\geq 1$ with discriminant $D_X$ and conductor $N_X$. In the first part of the following proof of Theorem \ref{thm2} we adapt the method of Par{\v{s}}in \cite{parshin:shafarevich}, Oort \cite{oort:shafarevich} and de Jong-R\'emond \cite{jore:shafarevich}. In the second part we apply Proposition \ref{propd} which is based on results from Arakelov geometry of Faltings \cite{faltings:arithmeticsurfaces} and de Jong \cite{dejong:weierstrasspoints}.

\begin{proof}[Proof of Theorem \ref{thm2}]
We may and do assume that $g\geq 2$, and that $X$ is a semi-stable cyclic cover of prime degree. Thus there exists a finite morphism $\varphi:X\to\mathbb P^1_K$  of prime degree which is geometrically a cyclic cover. Let $q$ be the degree of $\varphi$, let $\mu_X$ be the quantity defined in Section \ref{secheights} and let $h_F(J)$ be the Faltings height of the Jacobian $J$ of $X$. Proposition \ref{proph} (i) gives \begin{equation}\label{eq:hfj}
h_F(J)\leq 2^{2^{22}9^g}\mu_X.
\end{equation}
To estimate $\mu_X$ we denote by $S$ the union of the finite places of $K$ where $X$ has bad reduction with the finite places of $K$ of residue characteristic $q$. Let $\mathcal R$ be the set from Section \ref{secheights}, let $L\in\mathcal R$, let $U=U(L)$ be the places of $L$ which divide a place in $T=S$ and let $\mu_U$ be the quantity from Section \ref{secunits}. We observe
\begin{equation}\label{eq:muf}
\mu_X=\max(1,\mu_{U(L)})
\end{equation}
with the maximum taken over all $L\in\mathcal R$. We next show that $L$ is unramified outside $S$. Let $K(P)\subseteq L$ be the field of definition of a (geometric) ramification point $P$ of $\varphi$, and let $J_{K(P)}$ and $X_{K(P)}$ be the base changes of $J$ and $X$ to $K(P)$ respectively.
We denote by $\iota:X_{K(P)}\hookrightarrow J_{K(P)}$ the usual embedding defined over $K(P$) which maps $P$ to zero in $J_{K(P)}$. If $P'$ is a (geometric) ramification point of $\varphi$, then the divisor $(P')$ on $X_{K(P)}$, and the divisor $(\varphi(P'))$ on the projective line over $K(P)$, satisfy $\varphi^*(\varphi(P'))=q(P')$, since $q=\deg(\varphi)$. This implies that $\iota$ maps all (geometric) ramification points of $\varphi$  into the $q$-torsion points $J[q]$ of $J$. Let $K(J[q])$ be the field of definition of $J[q]$. The Jacobian $J$ of $X$ has good reduction outside $S$ and the set $S$ contains all finite places of $L$ with residue characteristic $q$. Thus \cite[Theorem 1]{seta:goodreduction} gives that $L$, which is contained in $K(J[q])$, is unramified outside $S$. We note that this would also follow from \cite[Lemme 2.1]{jore:shafarevich}. 

To bound the relative degree $l$ of $L$ over $K$ we let $\bar{K}$ denote an algebraic closure of $K$. The morphism $\varphi$ is defined over $K$ which provides that the absolute Galois group Gal$(\bar{K}/K)$ acts on the set of (geometric) ramification points of $\varphi$. The classical formula of Hurwitz implies that $\varphi$ has at most $2g+2$ (geometric) ramification points and that $q\leq 2g+1$.
This leads to
\begin{equation}\label{eq:lns}
l\leq 24g^4 \textnormal{ and } N_S\leq N_X(2g+1)^{d},
\end{equation}
where $N_S=\prod N_v$ with the product taken over all finite place $v\in S$. The field $L$ is unramified outside $S$. Thus an application of Proposition \ref{prope} (i), with $S=T$ and $r=9/8$, gives $\mu_U\leq k_3N_S^{4ld}D_K^{l}$, for $k_3$ an effective constant depending only on $d$ and $g$. This combined with (\ref{eq:muf}) and (\ref{eq:lns}) shows
\begin{equation}\label{eq:mu}
\mu_X\leq k_3N_S^{d(4g)^4}D_K^{(3g)^4}.
\end{equation}
Then we see that Proposition \ref{propd} (iv), (\ref{eq:hfj}) and (\ref{eq:lns}) give an upper bound for $\log D_X$ as stated in Theorem \ref{thm2}.

It remains to prove our claim (\ref{eq:thm2b}) and we now calculate an explicit constant $c_3$ for $q=2$. In the case $q=2$ we get that $X$ is a semi-stable hyperelliptic curve over $K$ of genus $g\geq 2$ and Proposition \ref{propd} (i) shows
\begin{equation*}
\log D_X\leq 2^{2^{23}9^g}\mu_X.
\end{equation*}
To obtain (\ref{eq:mu}) with an explicit $k_3$ we replace in the above arguments Proposition \ref{prope} (i) by Proposition \ref{prope} (ii). Then the displayed upper bound for $\log D_X$ shows the existence of a constant $c_3$ with the desired properties. This completes the proof of Theorem \ref{thm2}.
\end{proof}

We remark that if $X$ is a cyclic cover of prime degree, then the proof of Theorem \ref{thm2} shows in addition that the following statement holds. There exists an affine plane model $f(z_1,z_2)=0$ of $X$ over $\bar{K}$  that satisfies 
$$h(f)\leq cM_X^{d(4g)^8}, \ \  c=2^{d^2(5g)^9}D_K^{24g^4},$$  
for $M_X=\prod N_v$ with the product taken over all finite places $v$ of $K$ where $X$ has bad reduction.
To obtain an analogue of this statement for any curve over $K$ of genus at least one, the fairly general result \cite[Theorem 1.2]{bist:riemann} of Bilu-Strambi might be useful.

We now prove our results for curves $X$ of genus one. To improve our constants for such curves we combine Proposition \ref{propde} and Proposition \ref{prope} instead of applying Theorem \ref{thm1} (i) with $g=1$.

\begin{proof}[Proof of Theorem \ref{thm3}]
We suppose that $X$ has genus one. Let $E$ be the Jacobian of $X$ and let $\mu_E$ be the invariant from Section \ref{secheights}. Proposition \ref{propde} (iv)  and \cite[Proposition 5.1]{pesz:discriminant} give that the unstable discriminant $\gamma_X$ of $X$ satisfies $\log\gamma_X\leq 5\log N_X+24d\log 2$. Thus Proposition \ref{propde} (iii) shows
\begin{equation}\label{eq:logdx}
\log D_X\leq 6d\mu_E+5\log N_X+74d.
\end{equation}
To estimate $\mu_E$ we let $S$ be the union of the finite places of $K$ where $X$ has bad reduction with the finite places of $K$ of even residue characteristic. Let $L$ be the field of definition of the 2-torsion points of $E$ and let $U$ be the set of finite places of $L$ which divide a place in $S=T$. Then the quantity $\mu_U$ from Section \ref{secunits} satisfies
\begin{equation}\label{eq:mue}
\mu_E\leq \mu_U.
\end{equation}
The relative degree of $L$ over $K$ is at most $6$ and the criterion of N\'eron-Ogg-Shafarevich \cite[p.201]{silverman:aoes} provides that $L$ is unramified outside $S$. Hence applications of Proposition \ref{prope} (i) and (ii) with $S=T$ give estimates for $\mu_U$ which together with (\ref{eq:logdx}) and (\ref{eq:mue}) imply Theorem \ref{thm3}.
\end{proof}

Next, we show  Corollary \ref{corgenustwo} for genus two curves $X$. It is a direct consequence of Theorem \ref{thm1}, Theorem \ref{thm2} and Proposition \ref{propd} (iii).

\begin{proof}[Proof of Corollary \ref{corgenustwo}]
We assume that $X$ has genus two. Then \cite[Proposition 7.4.9]{liu:ag} gives that $X$ is a hyperelliptic curve over $K$. If $X$ is semi-stable, then the refinement (\ref{eq:thm2b}) of Theorem \ref{thm2} provides an upper bound for $D_X$ as stated. To treat the remaining case we use Proposition \ref{propd} (iii). It shows that $D_X$ is at most the minimal discriminant $\Delta_X$ of $X$. Then we see that the estimates in Theorem \ref{thm1} for $\Delta_X$ imply the statement. This completes the proof of Corollary \ref{corgenustwo}.
\end{proof}

Conditional on  $(abc)^*$ or $(abc)$, we now prove parts of the generalizations $(D)$ and ($\Delta$) of Szpiro's discriminant conjecture.

\begin{proof}[Proof of Theorem \ref{thm4}]
To prove (i) we let $g\geq 2$ be an integer and we assume that $(abc)^*$ holds for $n=24g^4$ with the constants $c,\kappa$. We take a hyperelliptic curve $C$ over $\QQ$ of genus $g$ with a $\QQ$-rational Weierstrass point.  Let $\mu_C$ be the quantity from Section \ref{secheights}, let $N_C$ be the conductor of $C$ and let $\Delta_C$ be the minimal discriminant of $C$. Proposition \ref{propd} (ii) gives
\begin{equation*}
\log \Delta_C\leq 8\nu^3(\mu_C+\log N_C+2\log(\nu)), \ \ \nu=2g+1.
\end{equation*}
Our assumption implies that $(abc)$ is true for $n$, $r=\kappa$ and $\epsilon=\kappa$ with the same constant $c$. Hence Proposition \ref{prope} (iii), (\ref{eq:sharp}) and the arguments in the proof of Theorem \ref{thm2} show that $\mu_C$ is bounded linearly in terms of $\log N_C$. This together with the above estimate for $\log \Delta_C$ gives that $(\Delta)$ holds for $C$ and we conclude (i).

On  combining the arguments which we used in the proofs of (i) and Theorem \ref{thm2}, we deduce statement (ii).

To show (iii) we take real numbers $r,\epsilon>1$. We assume that $(abc)$ holds for $n=6d,r,\epsilon$ with the constant $c$. In the sequel $k_9,\dotsc$ denote effective constants depending only on $d,c,r,\epsilon$. We suppose $X$ has genus one. Let $E$ be the Jacobian of $X$, let $\mu_E$ be the quantity from Section \ref{secheights} and let $n_v$ be the exponents of the minimal discriminant of $E$ defined in Section \ref{secconjectures}. Proposition \ref{propde} gives
\begin{equation}\label{eq:dxend}
\log D_X\leq 6d\mu_E+\sum_{v\in T_0}n_v \log N_v+6\sum_{v'\in T_1}\log N_{v'} +k_9,
\end{equation}
for $T_0$ the finite places $v$ of $K$ with $2\nmid N_v$ where $E$ has potential good reduction and for $T_1$  the finite places $v$ of $K$ with $2\nmid N_v$ where $E$ has reduction type $I_n^{*}$ for some $n\geq 1$. Let $S$ be the union of the finite places of $K$ where $E$ has bad reduction with the finite places of $K$ of even residue characteristic. In the proof of Theorem \ref{thm3} we showed that the field of definition $L$ of the 2-torsion points of $E$ is unramified outside $S$. Let $T_2$ be the finite places of $K$ where $E$ has reduction type $I_n$ (Kodaira) for some $n\geq 1$ and let $U$ be the finite places of $L$ which divide a place in $T_1\cup T_2$ or the rational integer 6. If $v$ is a finite place of $L$ which divides a place of $K$ where $E$ has not potential good reduction or which divides 2, then the classification \cite[p.448]{silverman:aoes} implies that $v\in U$. This shows that the set $T$ from Section \ref{secheights} is contained in $U$ and therefore the quantity $\mu_U$ from Section \ref{secunits} satisfies $\mu_E\leq\mu_U$. The relative degree of $L$ over $K$ is at most 6 and then our assumption on $(abc)$, Proposition \ref{prope} (iii) and (\ref{eq:sharp}) lead to
\begin{equation}\label{eq:mueend}
6d\mu_E\leq 6dr\sum_{v\in T_1\cup T_2}\log N_v+6\epsilon\log N_S+6\epsilon\log D_K+k_{10}, 
\end{equation}
where  $N_S=\prod N_v$ with the product taken over the places $v$ in $S$. For a finite place $v$ of $K$ let $f_v$ be the usual conductor exponent at $v$ of the abelian variety $E$. If $v\in T_0\cup T_1$ is a bad reduction place of $E$, then we get that $f_v\geq 2$ and thus the definition of $S$ combined with the classification \cite[p.448]{silverman:aoes} shows
\begin{equation}\label{eq:nsend}
6\epsilon\log N_S\leq 3\epsilon\sum_{v\in T_0\cup T_1}f_v\log N_v+6\epsilon \sum_{v'\in T_2}f_{v'}\log  N_{v'}+k_{11}.
\end{equation}
Any $v\in T_0$ satisfies $n_v\leq 5f_v$ by \cite[Proposition 5.1]{pesz:discriminant}. 
Hence, we see that (\ref{eq:dxend}), (\ref{eq:mueend}) and  (\ref{eq:nsend}) give that $\log D_X-(6\epsilon\log D_K+k_{12})$ is at most
$$
(3\epsilon+5)\sum_{v\in T_0}f_v\log N_v+3(dr+\epsilon+1) \sum_{v'\in T_1}f_{v'}\log N_{v'}+6(dr+\epsilon)\sum_{v''\in T_2}f_{v''}\log N_{v''}.
$$
We deduce statement (iii), and this completes the proof of Theorem \ref{thm4}.\end{proof}

Before we prove our second corollary, we remark that the proof of Theorem \ref{thm4} (iii) shows in addition that if $(abc)$ holds for $n=d,r,\epsilon$ with the constant $c$, then there is an effective constant $c'$, depending only on $d$, $r$, $\epsilon$ and $c$, with the following property. If $X$ is a genus one curve over $K$ and if the Jacobian of $X$ has four $K$-rational torsion points of order two, then
$
D_X\leq c'N_X^{6dr}D_K^{6\epsilon}.
$

\begin{proof}[Proof of Corollary \ref{corlowgenus}]
We replace Theorems \ref{thm1} and \ref{thm2} by Theorem \ref{thm4} in the proof of Corollary \ref{corgenustwo} and then the statement follows.
\end{proof}

{\bf Acknowledgements:} Most of the results were obtained when the author was a member at the IAS Princeton. He would like to thank the IAS for the friendly environment. In addition, he would like to thank Robin de Jong for answering a question, Prof. Szpiro for giving helpful explanations and for pointing out \cite{szpiro:deux,meoe:weilcurvediscriminants}, Prof. Deligne for providing his unpublished ``Lettre \`a Quillen'', and Ariyan Javanpeykar for motivating discussions.  This work was supported by the National Science Foundation under agreement No. DMS-0635607.

{\scriptsize
\bibliographystyle{amsalpha}
\bibliography{../../../literature}
}

\noindent IH\'ES, 35 Route de Chartres, 91440 Bures-sur-Yvette, France\\
E-mail adress: {\sf rvk@ihes.fr}

\end{document}